\documentclass[12pt,reqno,oneside]{amsart}
 \usepackage{verbatim,amsmath,amssymb,cite,xspace}
\usepackage{color,cite,graphicx}

\theoremstyle{plain}
\newtheorem{theorem}{Theorem}[section]
\newtheorem{lemma}{Lemma}[section]

\theoremstyle{definition}

\theoremstyle{remark}
\newtheorem{remark}{Remark}[section]

\newtheorem{claim}{{\bf Claim}}[section]

\usepackage{color}

\usepackage{amssymb}
\usepackage{amssymb, latexsym}
\usepackage{amsmath}
\usepackage{amscd}
\usepackage{enumerate}
\usepackage{amsfonts}
\usepackage{graphicx}
\usepackage{epsfig}

\newcommand{\st}[2]{L^{#1}_tL^{#2}_x}

\newcommand{\la}{\langle}
\newcommand{\ra}{\rangle}
\newcommand{\hl}{\dot{H}^1\times L^2}
\newcommand{\vp}{\phi}
\newcommand{\tl}{\tilde{\lambda}}
\newcommand{\tg}{\tilde{\gamma}}

\newcommand{\OR}{\overrightarrow}
\newcommand{\M}{\mathcal{M}}
\newcommand{\HL}{\dot{H}^1\times L^2}
\newcommand{\RHL}{\dot{H}_{{\rm rad}}^1\times L_{{\rm rad}}^2}
\def\R{\mathbb{R}}

\usepackage{color}

\renewcommand{\R}{\mathbb{R}}

\numberwithin{equation}{section}

\begin{document}
\title[]{Generic and non-generic behavior of solutions to defocusing energy critical wave equation with potential in the radial case}
\author{Hao Jia, Baoping Liu, Wilhelm Schlag, Guixiang Xu}

\begin{abstract}
In this paper, we continue our study  \cite{JiaLiuXu} on the long time dynamics of radial solutions to defocusing energy critical wave equation with a trapping radial potential in $3+1$ dimensions. For generic radial potentials (in the topological sense) there are only finitely many steady states which might be either stable or unstable.  We first observe that there can be {\it stable} excited states (i.e., a steady state which is not the ground state)  if the potential is large and attractive, although all small excited states are unstable. 
We prove that the set of initial data for which solutions 
scatter to  
any one unstable excited state forms a finite co-dimensional connected $C^1$ manifold in energy space. This amounts to the construction  of the global path-connected, and unique, center-stable manifold
associated  with, but not necessarily close to, any unstable steady state. In particular, the set of data for which solutions scatter to unstable states has empty interior in the energy space, 
and generic radial solutions scatter to one of the stable steady states.  Our main tools are (i) near any given finite energy radial initial data $(u_{0},u_{1})$ for which the solution $u(t)$ scatters to some unstable steady state $\phi$ we construct a $C^1$ manifold containing $(u_{0},u_{1})$ with the
property that any solution starting on the manifold scatters to $\phi$; moreover, any solution remaining near the manifold for all positive times lies on the manifold (ii) an exterior energy inequality from \cite{DKM,DKM1,JiaLiuXu}. The latter is used to obtain a result in the spirit  of the one-pass theorem~\cite{NS}, albeit with completely different techniques. 
\end{abstract}

\address{Hao Jia, Baoping Liu, Wilhelm Schlag: University of Chicago, Department of Mathematics, 5734 South University Avenue, Chicago, IL 60636, U.S.A.}
\email{jiahao@math.uchicago.edu, baoping@math.uchicago.edu, schlag@math.uchicago.edu}
\address{Guixiang Xu: Institute of Applied Physics and Computational Mathematics, Beijing, China}
\email{xu\_guixiang@iapcm.ac.cn }
 
\thanks{The third author was partially supported by the NSF through DMS-1160817.  The fourth author was partially supported by the NSF of China (No. 11171033, No. 11231006).}

\maketitle

\section{Introduction}

Fix $\beta>2$. Define
\begin{equation*}
Y:=\left\{V\in C(\R^3):\, V\,\,{\rm radial\,\,and}\,\,\sup_{x\in \R^3}(1+|x|)^{\beta}|V(x)|<\infty\right\},
\end{equation*}
with the natural norm
\begin{equation*}
\|V\|_Y:=\sup_{x\in \R^3}(1+|x|)^{\beta}|V(x)|.
\end{equation*}
We study solutions to 
\begin{equation}\label{eq:mainequation}
\partial_{tt}u-\Delta u-Vu+u^5=0,
\end{equation} 
with initial data $\OR{u}(0)=(u_0,u_1)\in \HL(\R^3)$. Since for a short time the term $Vu$ can be considered as a small perturbation, by adaptations of results in \cite{BaGe,Gri1,Gri2,Struwe} we know for any initial data $(u_0,u_1)\in \dot{H}^1\times L^2(\R^3)$, there exists a unique solution $${u}(t)\in C([0,\infty),\dot{H}^1)\cap L^5_tL^{10}_x([0,T)\times \R^3)$$ for any $T<\infty$ to equation (\ref{eq:mainequation}). Moreover, the energy
\begin{equation*}
\mathcal{E}(\overrightarrow{u}(t)):=\int_{\R^3}\Big[\frac{|\nabla u|^2}{2}+\frac{(\partial_tu)^2}{2}-\frac{Vu^2}{2}+\frac{u^6}{6}\Big](x,t)\,dx
\end{equation*}
is constant for all time. Our main interest in this work is the long time behavior of $u(t)$ under radial symmetry. If the operator $-\Delta -V$ has negative eigenvalues, then equation (\ref{eq:mainequation}) admits a nontrivial ground state $Q>0$, which is the global minimizer of 
\begin{equation*}
J(\phi):=\int_{\R^3}\Big[\frac{|\nabla \phi|^2}{2}-\frac{V\phi^2}{2}+\frac{\phi^6}{6}\Big]\, dx. 
\end{equation*}
It has negative energy. The linearized operator around $Q$ 
\begin{equation*}
\mathcal{L}_Q:=-\Delta-V+5Q^4
\end{equation*}
has no negative or zero eigenvalues, and no zero resonance. Consequently by well-known dispersive estimates for $\mathcal{L}_Q$ we know $Q$ is asymptotically stable: solutions with initial data close to $(Q,0)$ will scatter to $(Q,0)$. We remark that in our work by generic choice of potentials, all steady states are hyperbolic~\footnote{This means that the linearized operator around the steady state has neither zero eigenvalues nor zero resonance. } and consequently spectral stability implies asymptotic stability by well-known dispersive estimates for the associated linearized operator. Hence we will not distinguish the two notions below. In addition to the ground states $Q$ and $-Q$, there can be a number of  ``excited states"  with higher energies (see Appendix~A of \cite{JiaLiuXu}), which are changing sign steady states to equation (\ref{eq:mainequation}). Surprisingly, some of the excited states can be stable as well, although all small excited states can be shown to be unstable (see Section~\ref{sec:2} below). More precisely we have the following result. 

\begin{theorem}
There exists an open set $\mathcal{O}\subset Y$, such that for $V\in \mathcal{O}$, there exists an excited state $\phi$ to equation (\ref{eq:mainequation}) which is stable.
\end{theorem}

Roughly speaking, this is due to the stabilizing effect of the nonlinearity as a result of its defocusing nature, and the instability is mainly due to the potential. Hence if the excited state is large, the nonlinear stabilizing effect may dominate and the resulting dynamics around that excited state could become stable.

\smallskip 

Due to the presence of many steady states, in general the global dynamics can be quite complicated, even in the radially symmetric setting. 
\cite{JiaLiuXu} establishes  the following result characterizing the long time dynamics of  radial finite energy solutions.

\begin{theorem}\label{th:JLX}
Let $(u_0,u_1)\in \dot{H}^1\times L^2$ be radial. 
Denote 
\begin{equation}
\Sigma=\{(\phi,0)|\, (\phi,0) \,\,{\rm is\,\,a\,\,radial\,\,steady\,\,state\,\,solution\,\,to\,\,equation\,\,(\ref{eq:mainequation})}\}.
\end{equation}
Let $u\in C([0,\infty),\dot{H}^1)\cap L_t^5L_x^{10}([0,T)\times \R^3)$ for any $T<\infty$ be the unique solution to equation (\ref{eq:mainequation}) with initial data $(u(0),\partial_tu(0))=(u_0,u_1)$. Then for some radial finite energy solution $(u^L,\partial_tu^L)$ to the linear wave equation without potential \footnote{We often call such linear solutions {\it free radiation}.}
\begin{equation*}
({\rm LW})\,\,\,\,\partial_{tt}u-\Delta u=0,
\end{equation*}
we have
\begin{equation}
\lim_{t\to \infty}\inf_{(\phi,0)\in \Sigma}\|(u(t),\partial_tu(t))-(\phi,0)-(u^L(t),\partial_tu^L(t))\|_{\dot{H}^1\times L^2}=0.
\end{equation}
Moreover, for $V$ in a dense open set $\Omega\subset Y$, there are only finitely many radial steady states to equation (\ref{eq:mainequation}). In this case, there exist a steady state solution $(\phi,0)$ and some solution $(u^L,\partial_tu^L)$ to the linear wave equation without potential, such that
\begin{equation}
\lim_{t\to \infty}\|(u(t),\partial_tu(t))-(\phi,0)-(u^L(t),\partial_tu^L(t))\|_{\dot{H}^1\times L^2}=0.
\end{equation}
\end{theorem}

We remark that we can in fact choose the set $\Omega\subset Y$ such that for any $V\in \Omega$, all   steady states are hyperbolic~\footnote{In Theorem 6.1~\cite{JiaLiuXu}, we only showed that in the radially symmetric case, the linearized operator has neither a zero eigenvalue nor a zero resonance when restricted to radial functions. This leaves the possibility of having zero eigenvalue or zero resonance when we consider nonradial functions. We will address this issue in Section 2.}. We fix this choice of $\Omega$ below. 
Theorem \ref{th:JLX} is a particular instance of the {\it soliton resolution conjecture} for general dispersive equations, which has been intensively studied for many dispersive equations. We refer the reader to \cite{DKM1,DKM} and references therein for results on the focusing energy critical wave equation, \cite{KLS,Cotemap1,Cotemap2,cotesoliton,KLIS2} and references therein for results on equivariant wave maps, and \cite{TaoCompact,TsaiYau} for results on Schr\"{o}dinger equation with potential.   The difference between these works on this lies with the {\em defocusing nature} of our equation which precludes any blowup. In other words,
the flow on phase space   is  global in time, and together with  \cite{JiaLiuXu}   the present  work establishes a complete description of the long term dynamics as well as a decomposition of
the global data set into components which lead to distinct final states. 

\smallskip 

The result in \cite{JiaLiuXu} proves convergence for all  radial solutions, thus establishing the so-called {\em soliton resolution} in the setting of equation (\ref{eq:mainequation}). The proof relies crucially on the channel of energy inequalities for the linear wave equation, introduced in the works of Duyckaerts, Kenig and Merle \cite{DKM1,DKM}. This tool implies, amongst other properties, that all non-stationary  radial solutions  emit a positive amount of energy into large distances (the ``far field"). The main local decay mechanism for equation (\ref{eq:mainequation}) is the dispersion of energy into large distances, and the channel of energy inequalities provide a powerful tool to quantify such effects. In fact, due to the presence of the potential which destroys many of the favorable algebraic identities of virial type, \footnote{The virial type identities can still be of some use even in this context, see \cite{TaoCompact}.} the channel of energy inequality is perhaps the only tool currently available to measure dispersion in this context. As a consequence, in absence of radial symmetry, where  the channel of energy inequalities (see\cite{DKMnonradial}) become less effective, we have little knowledge of the ``compact solutions", i.e., solution $u(t)$ with the property that $\{\vec{u}(t), t\in \R\}$ is pre-compact in $\dot{H}^1\times L^2$.  This should be compared to the focusing energy critical wave equation, for which one knows (albeit along a sequence of times)  modulo symmetries, that compact solutions converge to some steady state. \footnote{A recent result of Duyckaerts, Kenig and Merle \cite{DKM3}, under certain non-degeneracy assumptions,   completely characterizes all compact solutions as Lorentz transformations of steady states.}

\smallskip

In this paper our main goal is to obtain refined descriptions of the global dynamics of solutions to equation (\ref{eq:mainequation}) in the radial case. 
Let us denote
\begin{equation*}
\dot{H}_{{\rm rad}}^1\times L^2_{{\rm rad}}:=\left\{(u_0,u_1)\in\HL(\R^3):\,\,(u_0,u_1)\,\,{\rm radial}\right\}.
\end{equation*}
We establish the following result. 

\begin{theorem}\label{th:maintheoremintro}
Let $\Omega$ be an open dense subset of $Y$ such that equation (\ref{eq:mainequation}) has only finitely many steady states,  
 which are all hyperbolic, and let $\Sigma$ be the set of radial  steady states. 
 Denote $\OR{u}(t):=\OR{S}(t)(u_0,u_1)$ as the solution to equation (\ref{eq:mainequation}) with radial initial data $(u_0,u_1)\in\RHL(\R^3)$. 
 For each $(\phi,0)\in\Sigma$, define
\begin{equation}
\M_{\phi}:=\left\{(u_0,u_1)\in \RHL(\R^3):\,\OR{S}(t)(u_0,u_1)\,\,{\rm scatters\,\,to\,\,}(\phi,0)\,{\rm as}\,\,t\to +\infty\right\}.
\end{equation}
Denote
\begin{equation}
\mathcal{L}_{\phi}:=-\Delta -V+5\phi^4
\end{equation}
as the linearized operator around $\phi$. If $\mathcal{L}_{\phi}$ has no negative eigenvalues, then $\M_{\phi}$ is an open set $\subseteq \RHL(\R^3)$. 
If $\mathcal{L}_{\phi}$ restricted to radial functions has $n$ negative eigenvalues, then $\M_{\phi}$ is a path connected $C^1$ manifold $\subset \RHL(\R^3)$ of codimension $n$.   
\end{theorem}

\smallskip
\noindent
{\it Remark.} This result shows that each unstable excited steady state attracts a finite co-dimensional manifold of solutions, hence scattering to unstable excited states is non-generic. If $\mathcal{L}_{\phi}$ has no negative eigenvalues, then $\phi$ is stable. This is a relatively straightforward consequence of the known dispersive estimates for $\mathcal{L}_{\phi}$ (see \cite{Marius}) \footnote{Due to our relatively mild decay assumption on $V$, the dispersive estimates we need are close to optimal, hence the need for the work~\cite{Marius} which requires less decay on the potential than, say, $|V(x)|\lesssim \frac{1}{(1+|x|)^{5+}}$ that is usually required for the $L^p$ boundedness of wave operators in some other works.} and standard perturbation arguments. 

On the other hand, if $\mathcal{L}_{\phi}$ has negative eigenvalues then the local dynamics near $(\phi,0)$ is nontrivial. Thanks to \cite{SchNak} and reference therein, it is now well-known in a small neighborhood (in the energy space) of $(\phi,0)$ that we can construct a center-stable  manifold, on which the solution scatters to $(\phi,0)$. Off that manifold, the solution will exit the small neighborhood in finite time. In particular, this center-stable  manifold is unique. Generally speaking, after exiting the small neighborhood, we lose control on the dynamics based on perturbative arguments alone and some {\em global information} is needed. While the {\bf one-pass theorem } provided this global information in~\cite{SchNak}, here it is the channel of energy inequality that allows for the key global control on the solution after the exit time. We will provide further explanations below. 

\smallskip 

Let us briefly outline the main ideas in the proof of Theorem~\ref{th:maintheoremintro}. Take any unstable steady state $(\phi,0)\in \Sigma$ and a radial finite energy solution $\OR{u}(t)$ with initial data $(u_0,u_1)$ which scatters to $(\phi,0)$, i.e., for some radial  solution $\OR{u}^L(t)$ to the linear wave equation (LW), we have
\begin{equation}\label{eq:convergenceinenergy}
\lim_{t\to\infty}\|\OR{u}(t)-\OR{u}^L(t)-(\phi,0)\|_{\HL(\R^3)}=0.
\end{equation}
We first show in a small neighborhood $B_{\epsilon}((u_0,u_1))\subset \RHL(\R^3)$ there exists a local manifold $\M$, such that any solution $\OR{v}(t)$ with initial data on this manifold remains close to $\OR{u}(t)$ in $\HL$ for all positive times and also scatters to $(\phi,0)$. 
Moreover, this manifold has the following uniqueness property:  any radial finite energy solution $\OR{v}(t)$ which stays close to $\OR{u}(t)$ for all positive times necessarily emanates from  $\mathcal{M}$. The construction of this manifold differs from the usual ones in that this is not a center-stable manifold around a steady state. In fact, since the energy of the solution $\OR{u}(t)$ may be much higher than that of $(\phi,0)$, the free radiation $\OR{u}^L$ may contain a large amount of excess energy. One new technical aspect is that in addition to using (\ref{eq:convergenceinenergy}), we also need the space-time control on the radiation term, such as 
\begin{equation}\label{eq:strichartzradiation}
u-\phi\in L^5_tL^{10}_x( [0,\infty)\times \R^3).
\end{equation}
(\ref{eq:strichartzradiation}) is of course expected, but was not usually mentioned in the literature. With the help of (\ref{eq:strichartzradiation}), the construction of $\M$ follows from standard techniques. The next step is to describe the dynamics of solutions starting in $B_{\epsilon}((u_0,u_1))\subset \RHL(\R^3)$, but off the manifold $\M$. This is where we need the global control provided by the channel of energy inequalities. Take any solution $\OR{v}(t)$ starting in $B_{\epsilon}((u_0,u_1))\subset \RHL(\R^3)$ and off the manifold (possibly with a smaller $\epsilon$), then by the property of $\M$, $\OR{u}(t)-\OR{v}(t)$ 
will have energy of a fixed size at some time $t$, no matter how small $\OR{u}(0)-\OR{v}(0)$ is. We will show from this that $\OR{v}(t)$ will emit a fixed amount {\it more} energy than $\OR{u}$, thanks to the channel of energy inequality. The main difficulty is that since $\OR{u}(t)$ may have already emitted a large amount of energy in order to settle down to $\phi$, we need to distinguish the new radiation from the old radiation. This is done with careful perturbation arguments as follows. Choose $(v_0,v_1)$ very close to $(u_0,u_1)$ so that the solutions $\OR{v}(t)$ and $\OR{u}(t)$ remain close for a sufficiently long time. During this time, the radiation has propagated sufficiently far from the origin (with the bulk of energy traveling at speed $\sim 1$). In the finite region, the solution $\OR{v}$ is just a small perturbation of $(\phi,0)$. Due to the assumption $(v_0,v_1)\not\in \M$, after another long time $\OR{v}(t)$ will deviate from $(\phi,0)$ in the finite region by a fixed amount. Then we apply the channel of energy inequality to show that $\OR{v}(t)$ emits a second piece of radiation, which is supported very far away from the first radiation. Hence, in total $\OR{v}(t)$ emits quantitatively more energy into spatial infinity although the energy of $\OR{v}(t)$ can be chosen arbitrarily close to that of $\OR{u}(t)$. Consequently, $\OR{v}(t)$ has less energy than $(\phi,0)$ in the finite region for large times, and must scatter instead to a steady state of lower energy, not $(\phi,0)$. This establishes the proposition that in a small neighborhood of $(u_0,u_1)$, only initial data on $\M$ can lead to solutions scattering to $(\phi,0)$. Thus $\M_{\phi}$ is truly a global manifold in $\RHL(\R^3)$ whence  Theorem \ref{th:maintheoremintro}. The fact that $\mathcal{M}_{\phi}$ is path connected follows from a perturbation argument which we present at the end of Section~\ref{sec:5}.

\smallskip 

\subsection{Some open questions}

Our investigation leaves open the question whether the finite co-dimensional manifold of radial finite energy data scattering to unstable steady states is closed in the energy topology. The answer to this question seems to be nontrivial and will require 
further understanding of the global dynamics. For example, consider an unstable excited state $(\phi,0)\in \Sigma$. It is not hard to show that
there is a radial  solution $\OR{u}(t)$ which converges to $(\phi,0)$ exponentially as $t\to-\infty$, i.e., $\OR{u}(t)$ is on the 
unstable manifold of $(\phi,0)$, and hence $\mathcal{E}(\OR{u}(t))=\mathcal{E}((\phi,0))$. By the channel of energy property established below, $\OR{u}(t)$ will
emit a nontrivial amount of energy to large distances as $t\to +\infty$ and subsequently scatter to a steady state of strictly less energy, say $(\tilde\phi,0)$. 
However, there is a possibility that $(\tilde\phi,0)$ is also an excited state. In that case, denote by $\mathcal{M}_{\tilde{\phi}}$ the manifold of data scattering to
$(\tilde\phi,0)$ as $t\to +\infty$, we see $\OR{u}(t)\in \mathcal{M}_{\tilde{\phi}}$ for all $t$, but $\OR{u}(t)\to (\phi,0)$ in $\HL$ as $t\to-\infty$ and clearly $(\phi,0)\not\in\mathcal{M}_{\tilde{\phi}}$.
Consequently, in such a situation $\mathcal{M}_{\tilde{\phi}}$ would not be closed. Admittedly, such behavior should be non-generic (due to the fact that 
$\tilde\phi$ might be expected to be the ground state)  and perhaps impossible for a generic choice of $V$. We plan to address
this question in future work. 

\smallskip

Another interesting question is if this description of global dynamics can be achieved without the radial assumption. This question
seems to be very challenging. Recall that in the radial case, due to the channel of energy property, we only need to consider the dynamics
outside some well-chosen light cone, where the dynamics is relatively simple. In contrast,  in the non-radial case, where only less effective
channel of energy inequality is available, one must deal directly with the complicated dynamics in a finite region. In this case the only other global tool is the virial type identities. However
the presence of the spatial inhomogeneity $V$ seems to render such identities ineffective. In particular, we do not know if the only compact solutions
are steady states. Recall that $\OR{u}(t)$ is called {\it compact} if $\{\OR{u}(t):\,t\in \R\}$ is pre-compact in $\HL$. This is in sharp contrast with the energy critical focusing wave equation case, where one knows exactly what
these compact solutions are (modulo some non-degeneracy condition on steady states). Hence, a full characterization of compact solutions  seems to be a natural first step.

\medskip

Our paper is organized as follows. In Section~\ref{sec:2}, we study steady states to equation (\ref{eq:mainequation}) and show in particular the existence of stable excited steady states; 
in Section~\ref{sec:3}  we construct the local center-stable  manifold. The novelty of his construction lies with the fact that it is carried out near {\em any solution} which
scatters to a given unstable steady state, without, however, being necessarily close to the steady state in the energy topology. 
in Section~\ref{sec:4} we recall some results on the well-known profile decompositions and channel of energy inequalities, 
adapted to equation (\ref{eq:mainequation}); in Section~\ref{sec:5} we prove our main result Theorem~\ref{th:maintheoremintro}; Appendix~A contains some elliptic estimates for the steady states; Appendix~B proves an endpoint Strichartz estimate for the inhomogeneous wave equation in the radial case.

\section{Steady state solutions}\label{sec:2}

In this section we prove some results about the steady states that are relevant for the global dynamics. We first give  necessary and sufficient conditions for the existence of nontrivial ground state. Recall that such a state is the global minimizer of the energy functional
\begin{equation}
J(\phi):=\int_{\R^3} \Big[\frac{|\nabla \phi|^2}{2}-\frac{V\phi^2}{2}+\frac{\phi^6}{6}\Big]\, dx.
\end{equation}

\begin{lemma}\label{lm:globalminimizer}
Consider $J$ as a functional defined in $\dot{H}^1(\R^3)$. If the operator $-\Delta-V$ has negative eigenvalues then there exists a global minimizer $Q>0$ with $J(Q)<0$. If $-\Delta-V$ has no negative eigenvalues, then the only steady state solution $u\in\dot{H}^1(\R^3)$ to equation (\ref{eq:mainequation}) is $u\equiv 0$.
\end{lemma}

\begin{remark}
 The proof of this lemma is a simple application of variational arguments and the strong maximum principle, we omit the standard details.\end{remark} 

In the case that $-\Delta-V$ has no negative eigenvalues and assuming that we only consider radial solutions, then from the results in \cite{JiaLiuXu} we know that all radial finite energy solutions to equation (\ref{eq:mainequation}) scatter to the trivial steady state. In what follows we therefore assume that $-\Delta-V$ has some negative eigenvalues, so that we have nontrivial global minimizers $Q$ and $-Q$. We call $Q$ and $-Q$ {\it ground states}, and call other steady solutions {\it excited states}.\\

The next result shows the uniqueness of ground states. Note that we do not need radial symmetry here.
\begin{lemma}
There is at most one nontrivial nonnegative steady state in $\dot{H}^1(\R^3)$ to equation (\ref{eq:mainequation}).
\end{lemma} 
\begin{proof}
Suppose $Q_1$ and $Q_2$ are two nontrivial nonnegative steady solutions to equation (\ref{eq:mainequation}) in $\dot{H}^1(\R^3)$. Then
\begin{eqnarray*}
&&-\Delta Q_1-VQ_1+Q_1^5=0,\\
&&-\Delta Q_2-VQ_2+Q_2^5=0.
\end{eqnarray*}
By standard elliptic estimates, we have $Q_1,\,Q_2\in W_{{\rm loc}}^{2,p}(\R^3)$ for any $p<\infty$. Moreover we have the following decay estimates
\begin{equation}
|Q_1(x)|+|Q_2(x)|\leq \frac{C}{1+|x|},\quad x\in \R^3.
\end{equation}
The above claim on the regularity and decay holds for any steady state in $\dot{H}^1$, and follows from more or less standard elliptic techniques. For the sake of completeness, we provide a proof in the Appendix A. By Strong Maximum Principle, we see that $Q_1,\,Q_2>0$. Denote the open set
\begin{equation*} 
\Omega:=\{x: \,Q_1(x)>Q_2(x)\}, 
\end{equation*}
we have
\begin{equation}\label{eq:compare}
\int_{\Omega}Q_2\left(-\Delta Q_1-VQ_1+Q_1^5\right)-Q_1\left(-\Delta Q_2-VQ_2+Q_2^5\right)dx=0.
\end{equation}
By the regularity and decay properties of $Q_1,\,Q_2$ we can integrate by parts in equation (\ref{eq:compare}), noting that $Q_1=Q_2$ on $\partial \Omega$, we see that
\begin{equation}\label{eq:contradiction}
\int_{\partial\Omega} Q_1\frac{\partial}{\partial n}(Q_2-Q_1)\,d\sigma+\int_{\Omega} Q_1Q_2(Q_1^4-Q_2^4)\,dx=0.
\end{equation}
Note that 
\begin{equation*}
\frac{\partial}{\partial n}\left( Q_2-Q_1\right)\ge 0 \quad {\rm on\,\,\,}  \partial\Omega,
\end{equation*}
 and
\begin{equation*}
 Q_1Q_2\left(Q_1^4-Q_2^4\right)>0 \quad {\rm in}\,\,\, \Omega.
\end{equation*}
 Thus equation (\ref{eq:contradiction}) can hold only if $\Omega=\emptyset$. Thus $Q_1 \leq Q_2$. Similarly $Q_2\leq Q_1$. Therefore $Q_1\equiv Q_2$.\end{proof}

Naively one might expect excited states to be unstable, since they change sign. However in general this may not be the case, as seen from the following theorem.
\begin{theorem}\label{th:stableexcitedstate}
There exists an open set $\mathcal{O}\in Y$ such that for any $V\in \mathcal{O}$, there exists an exicted state ${\phi}$ to equation (\ref{eq:mainequation}) which is stable.
\end{theorem}

\begin{proof}The proof is based on simple perturbation arguments, once a good large potential is chosen. We can  construct an excited state near a good ``profile", 
for which the linearized operator is explicit and stable, with a well-chosen potential. Then we can conclude that the linearization near the excited state is also stable. 

\noindent
{\it Step 1: choice of $V$}. Denote 
\begin{equation*}
W:=\frac{1}{\left(1+\frac{|x|^2}{3}\right)^{\frac{1}{2}}}
\end{equation*}
as the unique (up to scaling and sign change) radial $\dot{H}^1(\R^3)$ solution to 
\begin{equation}
-\Delta u=u^5.
\end{equation}
 Let us take 
\begin{equation}
V_1(x)=2W^4,
\end{equation}
and positive $\lambda$ sufficiently large to be chosen below. Set 
\begin{equation}\label{eq:specialpotential}
V_{1\lambda}(x)=\lambda^2V_1(\lambda x). 
\end{equation}
It is easy to check that $W$ solves 
\begin{equation*}
-\Delta u-V_1u+u^5=0,
\end{equation*}
and that $W_{\lambda}(x)=\lambda^{\frac{1}{2}}W(\lambda x)$ solves
\begin{equation*}
-\Delta u-V_{1\lambda}u+u^5=0.
\end{equation*}
We choose $V:=V_1+V_{1\lambda}$.

\noindent

\medskip

{\it Step 2: construction of a stable excited state.} Consider the following elliptic equation
\begin{equation}\label{eq:stableexcitedstate}
-\Delta \phi-V\phi+\phi^5=0.
\end{equation}
Our goal is to show if $\lambda$ is sufficiently large then we can construct a steady state $\phi$ of the form 
\begin{equation}\label{eq:approximate}
\phi=W-W_{\lambda}+\eta,
\end{equation}
with some small $\eta$. The equation for $\eta$ is
\begin{equation}\label{eq:auxequation}
-\Delta\eta+(3W^4+3W_{\lambda}^4-20WW_{\lambda}^3-20W_{\lambda}W^3+30W^2W_{\lambda}^2)\eta+N(\eta,\lambda)=f_{\lambda},
\end{equation}
where the nonlinear term is
\begin{equation}
N(\eta,\lambda)=10(W-W_{\lambda})^3\eta^2+10(W-W_{\lambda})^2\eta^3+5(W-W_{\lambda})\eta^4+\eta^5,
\end{equation}
and the nonhomogeneous term is
\begin{equation}
f_{\lambda}=-(W-W_{\lambda})^5+W^5-W_{\lambda}^5-V_1W_{\lambda}+WV_{1\lambda}.
\end{equation}
If $\lambda$ is sufficiently large, $f_{\lambda}$ will be small in appropriate function spaces and we can use a perturbation argument to solve for $\eta$. A key ingredient is the following standard uniform estimate in $\lambda$ on the linear part:

\smallskip 

\begin{claim}\label{lm:linearoperator}
For sufficiently large $\lambda$, the operator $$L_{\lambda}:=-\Delta+3W^4+3W_{\lambda}^4-20WW_{\lambda}^3-20W_{\lambda}W^3+30W^2W_{\lambda}^2:\,\, \dot{H}^1(\R^3)\to  \dot{H}^{-1}(\R^3)$$ is invertible and we have the following estimate on the norm of the inverse operator $L_{\lambda}^{-1}$
\begin{equation}\label{eq:uniformbound}
\|L_{\lambda}^{-1}\|_{\dot{H}^{-1}(\R^3)\to  \dot{H}^{1}(\R^3)}\leq C,
\end{equation}
where $C$ is some absolute constant.
\end{claim}

\smallskip

\noindent
{\it Proof of Claim \ref{lm:linearoperator}.} It is easy to verify that $-\Delta:\, \dot{H}^1(\R^3)\to  \dot{H}^{-1}(\R^3)$ is invertible, and that $L_{\lambda}$ is a compact perturbation of $-\Delta$. Thus to prove that 
$L_{\lambda}$ is invertible we only need to show its kernel is trivial. This follows directly from the following bound for large $\lambda$
\begin{equation}\label{eq:aprioribound}
(L_{\lambda}\phi,\phi)\ge \frac{1}{2}\|\phi\|^2_{\dot{H}^1(\R^3)}, \,\,\forall \phi\in  \dot{H}^1(\R^3),
\end{equation}
where the inner product is with respect to the $\dot{H}^1$ and $\dot{H}^{-1}$ pairing. The proof of   (\ref{eq:aprioribound}) is an easy consequence of integration by parts argument and H\"{o}lder's inequality, once we note 
that 
\begin{equation*}
\lim_{\lambda\to\infty}\|WW_{\lambda}^3+W^3W_{\lambda}+W^2W_{\lambda}^2\|_{L^{3/2}(\R^3)}=0.
\end{equation*}
Suppose $L_{\lambda}u=f$. 
Then from the bound (\ref{eq:aprioribound}) and H\"older's inequality we infer that  
\begin{equation}
\|\nabla u\|_{L^2(\R^3)}\leq 2 \|f\|_{\dot{H}^{-1}},\,\,\,{\rm for\,\,sufficiently\,\,large\,\,}\lambda.
\end{equation}
In view of the preceding, the bound (\ref{eq:uniformbound}) follows immediately.

\medskip

Using the uniform bound on $L_{\lambda}^{-1}$ we can easily solve for $\eta$.
\begin{claim}\label{eq:existence}
For any $\epsilon>0$, if $\lambda$ is sufficiently large, then there exists a solution $\eta\in \dot{H}^1\cap L^6(\R^3)$ to equation (\ref{eq:auxequation}), in the sense of distributions, with 
\begin{equation}
\|\eta\|_{L^6(\R^3)}<C\epsilon,
\end{equation}
where $C$ is an absolute constant.
\end{claim}

\smallskip

\noindent
{\it Proof.} Take small $\epsilon>0$. We can take $\lambda$ sufficiently large, so that $\|f_{\lambda}\|_{L^{6/5}}<\epsilon$ and (\ref{eq:uniformbound}) holds.
 We  reformulate equation (\ref{eq:auxequation}) in the following way
\begin{equation}\label{eq:contraction}
\eta=L_{\lambda}^{-1}f_{\lambda}-L_{\lambda}^{-1}N(\eta,\lambda).
\end{equation}
Note that $L^{6/5}(\R^3)$ embeds continuously into $\dot{H}^{-1}(\R^3)$, thus the right hand side makes sense. Now one can check that
$$L_{\lambda}^{-1}f_{\lambda}-L_{\lambda}^{-1}N(\eta,\lambda)$$ is a contraction mapping in $B_{2\epsilon}\subseteq L^6(\R^3)$, if we choose $\epsilon$ small enough. Thus equation (\ref{eq:contraction}) and consequently equation (\ref{eq:auxequation}) has a unique solution $\eta$, with $ \|\eta\|_{L^6(\R^3)}\leq 2\epsilon$. Clearly, this $\eta$ satisfies the requirement of  Claim~2.2.

\smallskip

By looking at the $L^6$ norm of the positive and negative parts of $\phi$ in (\ref{eq:approximate}), it follows that the steady state $W-W_{\lambda}+\eta$ changes sign if we choose $\epsilon$ sufficiently small, and thus is an excited state. Moreover the linearized operator around this excited state is 
\begin{equation*}
-\Delta-2W^4-2W_{\lambda}^4+5(W-W_{\lambda}+\eta)^4.
\end{equation*}
If we choose $\epsilon$ small enough one can show this operator is nonnegative and has no negative eigenvalues nor zero eigenvalues/resonance. 
A standard local perturbation analysis implies that the excited state $W-W_{\lambda}+\eta$ is stable.

\smallskip 

\noindent

{\it Step 3: stability with respect to $V=V_1+V_{1\lambda}$}.  We now show that our construction is stable with respect to small perturbations of the potential $V$ in $Y$. This is more or less clear from the existence proof. Below we just outline some key points. We now formulate

\begin{claim}\label{lm:excitedstatesforgenericpotential}
Let $V:=V_1+V_{1\lambda}$ be defined as above. Assume $\lambda$ sufficiently large and $\delta$ sufficiently small, then for any radial potential $\widetilde{V}\in Y$ satisfying
\begin{equation}
\|\widetilde{V}-V\|_Y\leq \delta,
\end{equation}
there exists a stable excited state to 
\begin{equation}
-\Delta\phi-\widetilde{V}\phi+\phi^5=0\,\,\,\,{\rm in}\,\,\R^3.
\end{equation}
\end{claim}

\noindent
{\it Proof of Claim \ref{lm:excitedstatesforgenericpotential}.} For any $\widetilde{V}\in Y$ with 
\begin{equation*}
\|\widetilde{V}-V\|_Y\leq \delta,
\end{equation*}
for some $\delta$ sufficiently small. We have
\begin{equation*}
\|(-\Delta-\widetilde{V})-(-\Delta -V)\|_{\dot{H}^1\to\dot{H}^{-1}}\leq C \delta.
\end{equation*}
Hence if we choose $\lambda$ sufficiently large, $\delta$ sufficiently small, we can repeat {\it Step 2} with $V$ replaced by $\widetilde{V}$. The resulting $L_{\lambda}$ will still satisfy the uniform bound in Claim \ref{lm:linearoperator}, $f_{\lambda}$ with extra terms $(\widetilde{V}-V)W$ and $(\widetilde{V}-V)W_{\lambda}$, and $N(\eta,\lambda)$ can still be controlled in exactly the same way, as long as $\delta$ is chosen sufficiently small. One can then use the  Contraction Mapping Theorem to finish the proof. We omit the routine details.
This finishes the proof of Theorem~\ref{th:stableexcitedstate}.

\end{proof}

The above theorem proves the existence of stable excited states, on the other hand it is clear that there are unstable excited states. For example, suppose that the operator $-\Delta-V$ has negative eigenvalues (so that there are nontrivial ground states), then the excited state $\phi\equiv 0$ is unstable. A perhaps more interesting fact is that ``newly bifurcated" excited states are unstable.

\begin{lemma}\label{lm:smallexcitedstate}
Let $\alpha\in [0,\infty)$, $V\in C_c^{\infty}(\R^3)$ be nonnegative and not identically zero. Suppose that for $\alpha\ge\alpha_1$ the principal eigenvalue $\lambda_1(\alpha)$ of the operator $-\Delta-\alpha V$ is negative. Assume further that $$(\alpha_2,0)\in (\alpha_1,\infty)\times \dot{H}^1(\R^3)$$ is a bifurcation point for the equation 
\begin{equation}\label{eq:dynamicalequation}
-\Delta \phi-\alpha V\phi+\phi^5=0,
\end{equation}
in the sense that for any $\epsilon>0$ there is a nontrivial radial steady state $(\alpha,\phi)$ to equation (\ref{eq:dynamicalequation}) with 
\begin{equation}\label{eq:smallexcitedstates}
|\alpha-\alpha_2|+\|\phi\|_{\dot{H}^1(\R^3)}\leq \epsilon.
\end{equation}
Then for $\epsilon$ sufficiently small, the steady state $(\alpha,\phi)$ satisfying (\ref{eq:smallexcitedstates}) is unstable.
\end{lemma}
\begin{proof}Choosing $\epsilon$ sufficiently small we have $\alpha>\alpha_1$. Since $V$ is nonnegative, we have that $$\lambda_1(\alpha)\leq\lambda_1(\alpha_1)<0,$$ thus there exist $\delta>0$ and $\psi\in\dot{H}^1(\R^3)$ such that
\begin{equation}
\int_{\R^3}(|\nabla\psi|^2-\alpha V\psi^2)\, dx\leq\int_{\R^3} (|\nabla\psi|^2-\alpha_1 V\psi^2)\, dx\leq -\delta \|\psi\|^2_{\dot{H}^1(\R^3)}.
\end{equation} 
Note the linearized operator around $\phi$ is
\begin{equation}
-\Delta-\alpha V+5\phi^4.
\end{equation}
If $\epsilon$ is sufficiently small, by H\"older's inequality, we obtain
\begin{equation*}
\int_{\R^3}(|\nabla \psi|^2-\alpha V\psi^2+5\phi^4\psi^2)\, dx\leq -(\delta-C\epsilon^4) \|\psi\|^2_{\dot{H}^1}<0.
\end{equation*}
Thus the linearized operator has at least one negative eigenvalue and consequently the small excited states are unstable. \end{proof}

\begin{remark}For a description of bifurcations, see Appendix A of \cite{JiaLiuXu}. The above arguments also imply that if $-\Delta-V$ has negative eigenvalues, then all sufficiently small excited states in $\dot{H}^1(\R^3)$ are unstable.
\end{remark}


In Theorem 6.1 of \cite{JiaLiuXu}, we showed that there exists a dense open set $\Omega_1\subseteq Y$ such that for any $V\in \Omega_1$, there exists only finitely many radial steady states to equation (\ref{eq:mainequation}), all of which are hyperbolic when the linearized operator is restricted to the space of radial functions. As is mentioned in the introduction, this leaves open the possibility that some radial steady states may still have zero eigenvalues or a zero resonance without the radial assumption. As the property of having zero eigenvalues or a zero resonance is non-generic, we can expect to eliminate such behavior by removing from $\Omega_1$ a closed set while making the remaining set still open and dense. More precisely we have the following result.
\begin{lemma}
Let $\Omega_1$ be a dense open subset of $Y$ such that for any $V\in \Omega_1$ there are only finitely many radial steady states to equation (\ref{eq:mainequation}), all of which are hyperbolic if the linearized operator is restricted to radial functions. Then there exists a dense open set $\Omega\subseteq \Omega_1$, such that for any $V\in \Omega$, all radial steady states are hyperbolic without radial symmetry.
\end{lemma}

\smallskip
\noindent
{\it Remark.} We fix this choice of $\Omega$ below.
\begin{proof} Take any $V\in \Omega_1$, and suppose that $\phi_1,\phi_2,\dots,\phi_n$ are the radial steady states corresponding to $V$. Since the linearized operator around any $\phi_i$ is hyperbolic in the space of radial functions, standard perturbation arguments imply that any sufficiently small perturbation of $V$ in $\Omega$ will not change the number of radial steady states, and each radial steady state $\phi_i$ depends smoothly on the perturbation. We will show that we can find a specific perturbation $\widetilde{V}$ of $V$ with arbitrarily small norm in $Y$, (in particular $\widetilde{V}\in\Omega_1$), such that the perturbed steady state $\phi_1(\widetilde{V})$ becomes hyperbolic even without restricting to radial functions. Since being hyperbolic is an open property, we can then make repeated small perturbations to the potential until all the steady states become hyperbolic even in the nonradial function space. Then it is clear the subset $\Omega\subseteq \Omega_1$ with the property that for any $V\in\Omega$, all the radial steady states are hyperbolic without restriction to radial functions, is dense and open in $Y$. Below we describe the perturbation in detail. For $\epsilon>0$ sufficiently small, set
\begin{equation}
\widetilde{V}:=V+\epsilon\, \phi_1^4.
\end{equation}
Suppose that the perturbed steady state becomes 
\begin{equation}
\widetilde{\phi_1}:=\phi_1+\epsilon\, \psi.
\end{equation}
Then the equation for $\psi$ is 
\begin{equation}\label{eq:perturbedsteadystate}
-\Delta\psi-V\psi+5\phi_1^4\psi=\phi_1^5+\epsilon\, \phi_1^4\psi-N(\psi,\epsilon),\quad{\rm in\,\,}\R^3,
\end{equation}
where 
\begin{equation}
N(\psi,\epsilon)=10\epsilon\, \phi_1^3\psi^2+10\epsilon^2\phi_1^2\psi^3+5\epsilon^3\phi_1\psi^4+\epsilon^4\psi^5.
\end{equation}
Noting that 
\begin{equation*}
-\Delta \phi_1-V\phi_1+5\phi_1^4\cdot\phi_1=4\phi_1^5,
\end{equation*}
and the fact that the linearized operator $-\Delta-V+5\phi_1^4$ is invertible from $\dot{H}^1\cap\dot{H}^2\hookrightarrow \dot{H}^{-1}\cap L^2$ when restricted to radial functions (which follows from hyperbolicity in radial functions), we can rewrite equation (\ref{eq:perturbedsteadystate}) as 
\begin{equation}\label{eq:invertedperturbedsteadystate}
\psi=\frac{\phi_1}{4}+(-\Delta-V+5\phi_1^4)^{-1}(\epsilon\,\phi_1^4\,\psi-N(\psi,\epsilon)).
\end{equation}
If we take $\epsilon$ sufficiently small we can assume $\widetilde{V}\in\Omega_1$, and moreover we can use standard perturbation arguments  to show that 
\begin{equation}
\psi=\frac{\phi_1}{4}+O_{\dot{H}^1\cap\dot{H}^2}(\epsilon).
\end{equation}
Then the linearized operator around the perturbed steady state $\phi_1+\epsilon\,\psi$ becomes
\begin{equation}
-\Delta-V+\left[5\left(1+\frac{\epsilon}{4}\right)^4-\epsilon\right]\phi_1^4+O_{L^{\frac{3}{2}}\cap L^{\infty}}(\epsilon^2).
\end{equation}
The key point for us is that 
\begin{equation}
\left[5\left(1+\frac{\epsilon}{4}\right)^4-\epsilon\right]\phi_1^4\ge 5\phi_1^4+4\epsilon\,\phi_1^4,
\end{equation}
hence we have gained a positive factor which will eliminate the zero eigenvalues/zero resonance. The proof is finished with the following claim and the Min-Max principle for eigenvalues.
\end{proof}

\begin{claim}\label{claimedlowerbound}
Let $V$ and $\phi_1$ be given as above. Suppose that the linearized operator $$-\Delta-V+5\phi^4_1$$ has $k\ge 0$ negative eigenvalues with corresponding eigenfunctions $\rho_1,\dots,\rho_k$, and possibly also zero eigenvalues or zero resonance. Then for $\epsilon>0$ sufficiently small, and any $f\in \dot{H}^1$ with 
\begin{equation*}
\int_{\R^3} f\,\rho_1\,dx=\cdots=\int_{\R^3} f\,\rho_k\,dx=0,
\end{equation*}
we have
\begin{equation}\label{eq:cruciallowerbound}
\int_{\R^3} (-V+5\phi_1^4+4\epsilon\,\phi_1^4)f^2+|\nabla f|^2\,dx\ge c\epsilon\,\|f\|^2_{\dot{H}^1},
\end{equation}
for some fixed $c>0$.
\end{claim}

\begin{proof} 
Consider the functional 
\begin{equation}
\Lambda (f):=\int_{\R^3} (-V+5\phi_1^4)f^2+|\nabla f|^2\,dx
\end{equation}
on the space 
\begin{equation}
X:=\left\{f\in \dot{H}^1:\,\int_{\R^3} f\,\rho_1\,dx=\cdots=\int_{\R^3} f\,\rho_k\,dx=0\right\}.
\end{equation}
Suppose 
\begin{equation*}
\left\{f\in\dot{H}^1:\,(-\Delta -V+5\phi_1^4)f=0\right\}={\rm span}\,\left\{Z_1,\dots,Z_m\right\}. 
\end{equation*}
Using the same arguments as in Proposition 3.6 in \cite{DKM3}, we can find linearly independent $E_1,\dots,E_m\in C_0^{\infty}$, with
\begin{equation}
\forall i=1\dots k,\,\forall j=1\dots m, \,\,\int \rho_i\,E_j=0,\,\, \forall i,j=1\dots m, \int \phi_1^4\,E_i\,Z_j=\delta_{ij},
\end{equation}
such that for any $f$ with
\begin{equation}
\int f\,\rho_i=0,\,\,\int f\,E_j\,\phi_1^4=0,\,\,\forall i=1\dots k, \,\,j=1\dots m,
\end{equation}
one has 
\begin{equation}\label{eq:nondegeneracybound}
\Lambda(f)\ge c\|f\|^2_{\dot{H}^1}.
\end{equation}
For any $f\in X$, we can decompose
\begin{equation}
f=\sum_{i=1}^mc_i \,E_i+g,\quad {\rm with}\,\,\int g\,Y_1=\dots=\int g\,Y_k=0,\,\,\,\int g\,E_1\,\phi_1^4=\dots=\int g\,E_m\,\phi_1^4=0.
\end{equation}
We distinguish two cases. \\

\noindent
{\it Case 1:} \, $\sum\limits_{i=1}^m|c_i|\ge \delta_1 \|f\|_{\dot{H}^1}$ for some small $\delta_1>0$. Then the claimed bound (\ref{eq:cruciallowerbound}) follows from
\begin{equation}
\Lambda(f)\ge 4\epsilon \int \phi_1^4\, f^2\ge \epsilon\int (\sum_{i=1}^m c_iE_i)^2\,\phi_1^4\,dx\ge\epsilon\sum_{i=1}^mc_i^2\ge \delta_2\delta_1\epsilon\, \|f\|^2_{\dot{H}^1},
\end{equation}
for some small $\delta_2>0$, where in the third inequality we have used the equivalence of norms in a finite dimensional space and the linear independence of $E_i$.\\

\noindent
{\it Case 2:}\, $\sum\limits_{i=1}^m|c_i|\ll \|f\|_{\dot{H}^1}$. Then the claimed bound (\ref{eq:cruciallowerbound}) follows directly from the bound (\ref{eq:nondegeneracybound}).
\end{proof}

\section{Construction of the local center-stable  manifold}
\label{sec:3}

In \cite{JiaLiuXu} we showed that there exists a dense open set $\Omega\subset Y$, such that for any potential $V\in\Omega$, there are only finitely many radial steady states and all radial finite energy solutions scatter to one 
of the steady states.  Furthermore, the linearized operator $-\Delta-V+5\vp^4$ has no zero eigenvalues nor zero resonance for any radial steady state $\vp$. Our goal in this section is to study local dynamics around a solution which scatters to an unstable excited state. Recall $$\OR{S}(t)=(S_0(t),S_1(t)), \quad t\in \R$$ denotes the solution flow. Thus given $(u_0,u_1)\in\dot{H}^1\times L^2(\R^{3})$, $\OR{S}(t)(u_0,u_1)$ is the solution to equation (\ref{eq:mainequation}) with initial data $(u_0,u_1)$. To prove our main result, let us first state the  Strichartz estimate for the linear wave equation in dimension 3.
\begin{theorem}\label{thm-Strichartz} 
Let $I$ be a time interval and let $v:I\times \R^3\rightarrow \R$ be a finite energy solution to the wave equation 
\[ (\partial_{tt}-\Delta) v =F\]
with initial data $(v(t_0), \partial_t v(t_0))= (f, g)$ for some $t_0\in I$. Then we have the estimates 
\begin{equation}\label{Strichartz}
\|(v,v_t)\|_{C_t^0 (\dot{H}^1\times L^2)}  + \|v\|_{\st{q}{r} (I\times \R^3)} \leq C(q,r) \left(\|(f,g)\|_{\hl}  +\|F\|_{\st{1}{2}(I\times \R^3)}\right),
\end{equation}
where $2\leq q \leq \infty$ and $2\leq r < \infty$ satisfy the scaling condition 
\begin{equation}\frac{1}{q}+\frac{3}{r}=\frac12 \label{scaling}\end{equation}
 and the wave admissible condition\begin{equation}  \quad \frac{1}{q} +\frac{1}{r}\leq \frac12.\label{wave-admissible}\end{equation}
 Moreover, if $f, g$ and $F$ are radial functions,  (\ref{Strichartz}) holds true when $(q,r)=(2,\infty)$.
 \end{theorem}
 Later, we will call a pair $(q,r), 2\leq q, r\leq \infty$ \textit{admissible} if it satisfies (\ref{scaling}), (\ref{wave-admissible}).

See~\cite{GV} for the proof of this theorem in the nonradial  and non-endpoint case.  
The forbidden endpoint 
  $(q,r)=(2,\infty) $ was found in~\cite{KM}, where it was also proved that for the homogeneous equation ($F\equiv 0$), $ (2,\infty)$ becomes ``admissible" if the initial data is of the form $(v(t_0), \partial_tv(t_0))=(0,g)$ with $g$
 radially symmetric.  For the sake of completeness, in Appendix B, we provide the proof of the  endpoint Strichartz estimate for the inhomogeneous equation with general radial data.  \\

For applications below, we need the following Strichartz estimates for solutions to linear wave equation with potential.
\begin{lemma}\label{lm:strichartzwithpotential}
Take $V\in Y$ such that the operator $-\Delta -V$ has no zero eigenvalues or zero resonance. Denote $P^{\perp}$ as the projection operator to the continuous spectrum of $-\Delta-V$. Denote
\begin{equation}\label{eq:sqrtoperator}
\omega:=\sqrt{P^{\perp}(-\Delta -V)}.
\end{equation}
Let $I$ be a time interval with $t_0\in I$. Then for any $(f,g)\in \HL(R^3)$ and $F\in L^1_tL^2_x(I\times\R^3)$, the solution $\OR{\gamma}(t)$ to the equation
\begin{equation}
\partial_{tt}\gamma+\omega^2\gamma=P^{\perp}F, \quad (t,x)\in I\times \R^3,
\end{equation}
with $\OR{\gamma}(t_0)=P^{\perp}(f,g)$ satisfies 
\begin{equation}\label{eq:strichartzwithpotential}
\|(\gamma,\gamma_t)\|_{C_t^0 (\dot{H}^1\times L^2)}  + \|\gamma\|_{\st{q}{r} (I\times \R^3)} \leq C(q,r) \left(\|(f,g)\|_{\hl}  +\|F\|_{\st{1}{2}(I\times \R^3)}\right),
\end{equation}
where $2\leq q \leq \infty$ and $2\leq r < \infty$ satisfy the conditions (\ref{scaling}) and (\ref{wave-admissible}).  Moreover, if $f, g$ and $F$ are radial functions,  (\ref{eq:strichartzwithpotential}) holds true when $(q,r)=(2,\infty)$.
\end{lemma}

\smallskip
\noindent
{\it Remark.} Our proof of Strichartz estimates (\ref{eq:strichartzwithpotential}) relies crucially on $L^p$ boundedness of wave operators with minimal decay conditions on $V$, obtained by Beceanu \cite{Marius}. For earlier important work on $L^p$ boundedness of wave operators, see Yajima\cite{Yajima} and references therein.
\begin{proof} The energy estimate on $\OR{\gamma}$ follows from standard integration by parts argument and inequality (\ref{eq:l2boundforomega}) below. We will therefore concentrate only on the Strichartz estimates.
Recall the definition of the forward wave operator
\begin{equation}
W_+:=s-\lim_{t\to\infty} e^{it(-\Delta-V)}e^{-it(-\Delta)}.
\end{equation}
It is well known that for $V\in Y$, $W_+$ is linearly isomorphic from $L^2$ to $P^{\perp}L^2$, see for example \cite{agmon2}. In \cite{Marius} Beceanu obtained among other things an important structural theorem for $W_+$, which implies that $W_+$ is also bounded in $L^p$ spaces with $1\leq p\leq \infty$:
\begin{equation}\label{eq:lpbound}
\|W_+\varphi\|_{L^p}\lesssim \|\varphi\|_{L^p}\quad \forall \varphi\in L^2\cap L^p.
\end{equation}
Hence $W_+$ can be naturally extended as a bounded operator in $L^p$ for $1\leq p<\infty$ and in $L^{\infty}_0$, which is the completion of $L^2\cap L^{\infty}$ in $L^{\infty}$. An important fact of $W_+$ is the intertwining property
\begin{equation}\label{eq:intertwine}
\varphi(\omega)=W_+\varphi(|\nabla|)W^{\ast}_+,
\end{equation}
which holds for $\varphi\in L^{\infty}$, and also for more general $\varphi$ by limiting arguments provided that one can obtain suitable bounds. Here $W^{\ast}_+$ is the adjoint operator of $W_+$. Note that $\OR{\gamma}(t)$ admits the following representation
\begin{eqnarray*}
\gamma(t)&=&\cos{(\omega t)}\,P^{\perp}f+\frac{\sin{\omega t}}{\omega} P^{\perp}g+\int_0^t\frac{\sin{\omega (t-s)}}{\omega} P^{\perp}F(s)\,ds\\
&=&W_+\cos{(|\nabla|t)}W_+^{\ast}\,f+W_+\frac{\sin{|\nabla| t}}{|\nabla|} W_+^{\ast}\,g+W_+\int_0^t\frac{\sin{|\nabla| (t-s)}}{|\nabla|}W_+^{\ast} P^{\perp}F(s)\,ds\\
&=&I+II+III.
\end{eqnarray*}
The Strichartz estimates for part II and III then follow directly from the $L^p$ boundedness (\ref{eq:lpbound}) of $W_+$ and the corresponding Strichartz estimates for free radiations. For the endpoint $(q,r)=(2,\infty)$, we only need to note in addition that $P^{\perp}$ leaves the space of radial functions invariant, since $V$ is radial. It remains to consider I. Firstly we claim
\begin{claim}\label{claim:linearisomorphism}
The operator $\omega$ which is initially defined in $H^2$, satisfies 
\begin{equation}\label{eq:l2boundforomega}
\|\nabla \varphi\|_{L^2}^2\lesssim\int(\omega \varphi)^2\lesssim \|\nabla \varphi\|_{L^2}^2,\quad\forall \varphi\in P^{\perp}H^2.
\end{equation}
and can be extended naturally as a linear isomorphism from $P^{\perp}\dot{H}^1$ to $P^{\perp}L^2$.
\end{claim}
Let us assume this claim momentarily. Then the Strichartz estimate for part I is easy to prove. By Claim \ref{claim:linearisomorphism} we can write $P^{\perp}f=\omega^{-1}\tilde{f}$ for some $\tilde{f}\in P^{\perp}L^2$. Thus 
\begin{eqnarray*}
W_+\cos{(|\nabla|t)}W_+^{\ast}\,f&=&W_+\cos{(|\nabla|t)}W_+^{\ast}\,\omega^{-1}\tilde{f}\\
                                                  &=&W_+\cos{(|\nabla|t)}|\nabla|^{-1}W_+^{\ast}\,\tilde{f},
\end{eqnarray*}
where we have used the fact that $W_+^{\ast}\omega^{-1}=|\nabla|^{-1}W_+^{\ast}$ on $P^{\perp}L^2$, which follows from the intertwining property (\ref{eq:intertwine}) by suitable limiting arguments with the help of bounds (\ref{eq:l2boundforomega}) and (\ref{eq:lpbound}). Note that $W_+^{\ast}\tilde{f}\in L^2$, consequently $|\nabla|^{-1}W_+^{\ast}\tilde{f}\in\dot{H}^1$. Hence the Strichartz estimates for part I follow straightforwardly from the corresponding estimates for free radiations and the bound (\ref{eq:lpbound}).
\end{proof}

Now we give a brief proof of Claim \ref{claim:linearisomorphism}. The second part of the inequality (\ref{eq:l2boundforomega}) is an easy consequence of the fact that $\omega$ is self adjoint and $\omega^2=P^{\perp}(-\Delta-V)$, and an integration by parts argument. The first part of the inequality follows from the assumption that $-\Delta-V$ has no zero eigenvalues or zero resonance. Hence $\omega$ can be extended as a bounded operator from $P^{\perp}\dot{H}^1$ to $P^{\perp}L^2$. Moreover $\omega$ has closed range due to the bound (\ref{eq:l2boundforomega}). Since $-\Delta-V$ has no zero eigenvalues, we conclude that the range of $P^{\perp}(-\Delta-V)$ is dense in $P^{\perp}L^2$. Thus the range of $\omega$ which is bigger than the range of $P^{\perp}(-\Delta-V)$ is also dense in $P^{\perp}L^2$. Combining these two facts, we see that $\omega$ is indeed a linear isomorphism from $P^{\perp}\dot{H}^1$ to $P^{\perp}L^2$ and the claim is proved.\\

\medskip  
  
  Our main goal in this section is to prove the following result.
\begin{theorem}\label{th:localmanifold}
Let $\Omega$ be a dense open subset of $Y$ such that equation (\ref{eq:mainequation}) has only finitely many radial steady states, all of which are hyperbolic. Suppose $V\in \Omega\subset Y$. Suppose $\overrightarrow{U}(t)$ is a radial finite energy solution to equation (\ref{eq:mainequation}) which scatters to an unstable steady state $(\phi,0)$. Let 
\begin{equation}
-k_1^2\leq-k_2^2\leq\cdots\leq-k_n^2<0
\end{equation}
be the negative eigenvalues of $-\Delta-V+5\vp^4$ restricted to radial functions (counted with multiplicity) with normalized eigenfunctions $\rho_1,\,\rho_2,\dots,\rho_n$, respectively. Decompose 
\begin{equation}
\RHL(\R^3)=X_s\oplus X_u,
\end{equation}
where 
\begin{equation}
X_s=\left\{(u_0,u_1)\in \RHL(\R^3):\,\langle k_ju_0+ u_1,\rho_j\rangle_{L^2}=0,\,{\rm for\,all}\,1\leq j\leq n\right\},
\end{equation}
and 
\begin{equation}
X_u={\rm span}\,\left\{(\rho_j,k_j\rho_j),\,\,1\leq j\leq n\right\}.
\end{equation}
Then there exist $\epsilon_0>0$, $T$ sufficiently large, a ball $B_{\epsilon_0}((0,0))\subset \RHL(\R^3)$, and a smooth mapping
\begin{equation}
\Psi: \overrightarrow{U}(T)+\left(B_{\epsilon_0}((0,0))\cap X_s\right)\longrightarrow \RHL,
\end{equation}
satisfying $\Psi(\overrightarrow{U}(T))=\overrightarrow{U}(T)$, with the following property. Let $\widetilde{\mathcal{M}}$ be the graph of $\Psi$ and set $\mathcal{M}=\OR{S}(-T)\widetilde{\mathcal{M}}$. Then any solution to equation (\ref{eq:mainequation}) with initial data $(u_0,u_1)\in\mathcal{M}$ scatters to $(\vp,0)$. Moreover, there is an $\epsilon_1$ with $0<\epsilon_1<\epsilon_0$, such that if a solution $\overrightarrow{u}(t)$ with initial data $(u_0,u_1)\in B_{\epsilon_1}(\overrightarrow{U}(0))\subset \RHL(\R^3)$ satisfies 
\begin{equation}
\|\overrightarrow{u}(t)-\overrightarrow{U}(t)\|_{\dot{H}^1\times L^2}<\epsilon_1 \,\,{\rm for\,\,all\,\,}t\ge 0,
\end{equation}
then $(u_0,u_1)\in \mathcal{M}$.
\end{theorem}

\begin{proof}By assumption, there exists  free radial radiation $\OR{U}^L$, such that
\begin{equation}\label{scatter-vp}
\lim_{t\to  \infty}\|\OR{U}(t)-(\phi,0)- \OR{U}^L(t)\|_{\dot{H}^1\times L^2}=0.
\end{equation}
We divide  our construction of the center-stable  manifold into a series of  steps:

\smallskip

\noindent
{\it Step 0: $L^6$ decay for free waves.} We observe that for any finite energy free radiation $\OR{u}^L$, we have  
\begin{equation} \|u^L(t)\|_{L^6_x} \rightarrow 0 \quad\quad \text{ as }\quad t\rightarrow \infty. \label{linear-small}\end{equation}
For smooth $\OR{u}^L(0)$ with ${\rm supp\,}\OR{u}^L(0)\Subset B_R$, we have 
\begin{equation*}
|u^L(t,x)|\leq \frac{C}{t}\chi_{t-R\leq|x|\leq t+R},\,\,\, {\rm for}\,\, t>R.
\end{equation*}
Then (\ref{linear-small}) follows by direct calculation. For general initial data, (\ref{linear-small}) follows from approximations by compactly supported smooth functions and the uniform bound $\|u^L(t)\|_{L^6(\R^3)}\leq C\|\OR{u}^L(0)\|_{\HL(\R^3)}$.

\smallskip

\noindent
{\it Step 1: space-time estimates for $U-\vp$.}
Denote $h(t,x) =U(t,x) -\vp(x)$, then $h$ satisfies the equation 
\begin{equation}h_{tt}-\Delta h -V(x)h + 5\vp^4 h +N(\vp, h)=0, \label{eq-h}\end{equation}
where 
\begin{equation*}
N(\vp, h) =(\vp+h)^5 -\phi^5 -5\vp^4h.
\end{equation*}
Since $\OR{U}\in L_t^\infty([0,\infty), \hl)$ and $U \in L^5_tL^{10}_x(I\times \R^3)$ for any finite interval $I$, by equation (\ref{eq:mainequation}) and Strichartz estimate (\ref{Strichartz}) we see $U\in L^2_tL^{\infty}_x(I\times \R^3)$. By standard elliptic estimates, we know that $\phi \in C^1(\R^3)$(see also Appendix A),  hence $h\in L^2_tL^{\infty}_x(I\times \R^3)$ for any finite time interval $I$.  In what follows, we will show that
\[\|h\|_{L^2_tL^{\infty}_x([0,\infty)\times \R^3)} <\infty.\]
Recall that $\rho_1, \ldots, \rho_n$ are the $n$ radial $L^{2}$ normalized orthogonal eigenfunctions of the operator $$\mathcal{L}_{\vp}=-\Delta - V + 5\vp^4,$$ corresponding to the eigenvalues (counting multiplicity) $-k_1^2\leq \ldots -k_n^2 <0$, respectively. From Agmon's estimate \cite{agmon}, we know these eigenfunctions decay exponentially. Writing
\begin{equation*}
h = \lambda_1(t)\rho_1 +\cdots + \lambda_n(t)\rho_n +\gamma,
\end{equation*}
with $\gamma\perp \rho_i$ for $ i=1,\cdots ,n$,  and plugging this  into equation (\ref{eq-h}) we obtain 
\begin{align}\label{eq:parametrizedequation}
 \sum_{i=1}^n(\ddot{\lambda}_i(t) -k_i^2 \lambda_i(t))\rho_i + \ddot{\gamma} +\mathcal{L}_{\vp}\gamma = N(\vp,h).
 \end{align}
 Denote by $P_i$   the projection operator onto the $i-$th eigenfunction and by $P^{\perp}$   the projection operator  onto the continuous spectrum restricted to radial functions~\footnote{Note that $P^{\perp}$ can be written via Stone's formula as integral of resolvant of $\mathcal{L}_{\phi}$, hence it is invariant for radial functions. In particular, $P^{\perp}$ does not involve any non-radial eigenfunctions.}, i.e., 
 \begin{align*}
 P_i =\rho_i \otimes \rho_i , \quad\quad
 P^{\perp} =I -\sum_{i=1}^n\rho_i \otimes \rho_i 
 \end{align*}
Applying the projection operators $P_i$ and $P^{\perp}$ to equation (\ref{eq:parametrizedequation}), we derive the following equations for $\lambda_i(t)$ and $\gamma(t,x)$: 
 \begin{equation}\label{system}\left\{\begin{aligned}
 \ddot{\lambda}_i(t) -k_i^2 \lambda_i(t) & = P_i N(\vp,h):=N_{\rho_i},\quad\quad i=1,\ldots n\\
 \ddot{\gamma} +\omega^2 \gamma &=P^{\perp} N(\vp,h) := N_c, \quad \quad\omega:=\sqrt{P^{\perp}\mathcal{L}_{\vp}}.
 \end{aligned}\right.\end{equation}
 Note that the steady state $\vp$ decays at the rate $O(\frac{1}{1+|x|})$ as $|x|\to\infty$, hence the potential in the operator $\mathcal{L}_{\vp}$ which is $-V+5\vp^4$, decays like $O(\frac{1}{(1+|x|)^{\min\{\beta,4\}}})$. This decay rate is better than the critical rate $O(\frac{1}{|x|^2})$ as $|x|\rightarrow \infty$ (in fact $-V+5\vp^4\in Y$). Hence we can  apply the result of Lemma \ref{lm:strichartzwithpotential} and conclude that Strichartz estimates as in Lemma \ref{lm:strichartzwithpotential} hold for solutions of the equation
\begin{equation}\label{eq:linearequationwithpotential}
v_{tt}+\omega^2 v =F,
\end{equation}
with $F$ radial and satisfying the compatibility condition $P^{\perp}F=F$. 

From (\ref{scatter-vp}) and (\ref{linear-small}), we know that $$\lim_{t\rightarrow \infty}\|h(t,x)\|_{L^\infty_tL^6_x([T,\infty)\times \R^3)} =0 $$ Also  using the fact that $\rho_i$ decay exponentially, we   have 
\[|\lambda_i(t)|=|\la \rho_i| h\ra |\leq \|\rho_i\|_{L^{\frac65}} \|h(t,x)\|_{L^6_x(\R^3)}\rightarrow 0 \hspace{1cm}\text{ as }t\rightarrow \infty .\]
Let $\Gamma(t)$ be the solution operator to the  equation $v_{tt}+\omega^2 v=0$, i.e., 
\[\Gamma(t-t_0)(\gamma(t_0), \dot{\gamma}(t_0)) = \cos (\omega (t-t_0))\gamma(t_0)+\frac{1}{\omega}\sin (\omega (t-t_0))\dot{\gamma}(t_0).\]
We claim:
\begin{claim}\label{Claim} Given any $\epsilon\ll 1$, we have  \[\|\Gamma(t-T)(\gamma(T), \dot{\gamma}(T))\|_{\st{2}{\infty}( [T,\infty)\times \R^3)} \leq \epsilon.  \]
for sufficiently large $T >0$.
\end{claim}
We postpone the proof of Claim~\ref{Claim} to the end of the proof of Theorem~\ref{th:localmanifold}.

Hence given a small positive number $\epsilon \ll 1$, which will be chosen later,  we can pick a large time $T=T(\epsilon,U)$, such that 
\begin{align}
\|h\|_{\st{\infty}{6}([T,\infty)\times \R^3)}& \leq \epsilon \label{h-small}\\
\|\lambda_i(t)\|_{L^\infty_t ([T,\infty))} &\leq \epsilon \hspace{1cm}    \label{small-lambda} \\
  \|\Gamma(t-T)(\gamma(T), \dot{\gamma}(T))\|_{\st{2}{\infty}( [T,\infty)\times \R^3)} &\leq \epsilon. \label{small-gamma}
\end{align}
From the equation for $\lambda_i(t)$ in (\ref{system}), we conclude that for $t\geq T$
\begin{align*}
\lambda_i(t)= & \cosh (k_i(t-T))\lambda_i(T)+\frac{1}{k_i}\sinh(k_i(t-T)) \dot{\lambda}_i(T)\\ &+\frac{1}{k_i}\int_T^t \sinh (k_i(t-s))N_{\rho_i}(s)\, ds\\
= & \frac{e^{k_i(t-T)}}{2}\left[\lambda_i(T)+\frac{1}{k_i}\dot{\lambda}_i(T) +\frac{1}{k_i}\int_T^t e^{k_i(T-s)}N_{\rho_i}(s)\, ds\right]+\mathcal{R}(t),
\end{align*}
where $\mathcal{R}(t)$ denotes a term  that remains bounded for bounded $N_{\rho_i}(s)$. By (\ref{small-lambda}), 
 we obtain the following stability condition  
\begin{equation}\label{condition} \dot{\lambda}_i(T)= - k_i\lambda_i(T)- \int_T^\infty e^{k_i(T-s)}N_{\rho_i}(s)ds. \end{equation}
Under this condition we can rewrite equation (\ref{system}) as the following integral equation 
\begin{equation}\label{system-1}\left\{\begin{aligned}
\lambda_i(t)&=e^{-k_i(t-T)}\left[\lambda_i(T)+\frac{1}{2k_i} \int_T^\infty e^{k_i (T-s)}N_{\rho_i}(s)ds\right]-\frac{1}{2k_i}\int_T^\infty e^{-k_i|t-s|}N_{\rho_i}(s)ds, \\
\gamma(t)&=\cos (\omega (t-T))\gamma(T)+\frac{1}{\omega}\sin (\omega (t-T))\dot{\gamma}(T)+\int_T^t\frac{\sin (\omega(t-s))}{\omega}N_c (s)ds. 
\end{aligned}\right.\end{equation}
For any time $\widetilde{T}>T$, we define 
\begin{align*} \|(\lambda_1,\ldots, \lambda_n, \gamma)\|_{X([T,\widetilde{T}))}:&= \sum_{i=1}^n\|\lambda_i(t)\|_{L_t^2([T, \widetilde{T}))} + \|\gamma\|_{\st{2}{\infty}([T,\widetilde{T})\times \R^3)}.
\end{align*}
With the help of Strichartz estimates from Lemma \ref{lm:strichartzwithpotential}, by estimating (\ref{system-1}), we get that
\begin{align}
\|\lambda_i(t)\|_{L^2([T, \widetilde{T}))}\leq &C_1\left( |\lambda_i(T)| + \|N_{\rho_i}\|_{L^{1}_t([T,\widetilde{T}))} + \|N_{\rho_i}\|_{L^{\infty}_t([\widetilde{T},\infty))}\right),  \label{lambda}\\
\|\gamma\|_{\st{2}{\infty}([T,\widetilde{T})\times \R^3)}\leq & C_2 \left(\|\Gamma(t-T)(\gamma(T), \dot{\gamma}(T))\|_{\st{2}{\infty}( [T,\widetilde{T})\times \R^3)} + \|N_c\|_{\st{1}{2}( [T,\widetilde{T})\times \R^3)}\right). \label{gamma}
\end{align}
Here that constant $C_1$ depends on the $L^1$ and $L^2$ integrals of $e^{-k_it}$ 
and the 
constant $C_2$ depends only on the constants in the Strichartz estimates.  

In (\ref{gamma}), instead of estimating initial data $(\gamma(T), \dot{\gamma}(T))$ in $\hl(\R^3)$ which may not be  small, we estimate its free evolution in   $\st{2}{\infty}( [T,\widetilde{T})\times \R^3)$. Consequently,  we obtain smallness because of (\ref{small-gamma}).

Using  the fact that $$N_{\rho_i}=\la \rho_i | N(\vp,h)\ra, \; N_{\rho}=\sum_{i} N_{\rho_i} \rho_{i},\; N_c=N-N_{\rho}$$ and  the exponential decay of $\rho_i$, we have
\begin{equation}\|N_{\rho_i}\|_{L^{1}_t([T,\widetilde{T}))},  \|N_c\|_{\st{1}{2}( [T,\widetilde{T})\times \R^3)}  \leq  \|N(\vp,h)\|_{\st{1}{2}([T,\widetilde{T})\times \R^3)}. \label{N-bound}\end{equation}
Recall that $$N(\vp, h) = 10 \vp^3h^2 + 10 \vp^2 h^3 +5 \vp h^4 +h^5.$$ Hence,  by H\"older inequalities, we have 
\begin{align*}
\|\vp^3 h^2\|_{\st{1}{2}([T,\widetilde{T})\times \R^3)} &\leq \|\vp\|_{L^6_x}^3 \|h\|^2_{\st{2}{\infty}([T,\widetilde{T})\times \R^3)},\\
\|\vp^2 h^3\|_{\st{1}{2}([T,\widetilde{T})\times \R^3)} &\leq \|\vp\|_{L^6_x}^2 \|h\|^2_{\st{2}{\infty}([T,\widetilde{T})\times \R^3)} \|h\|_{\st{\infty}{6}([T,\widetilde{T})\times \R^3)}, \\
\|\vp h^4\|_{\st{1}{2}([T,\widetilde{T})\times \R^3)} &\leq  \|\vp\|_{L^6_x} \|h\|^2_{\st{2}{\infty}([T,\widetilde{T})\times \R^3)} \|h\|^2_{\st{\infty}{6}([T,\widetilde{T})\times \R^3)}, \\ \|h^5\|_{\st{1}{2}([T,\widetilde{T})\times \R^3)}  &\leq    \|h\|^2_{\st{2}{\infty}([T,\widetilde{T})\times \R^3)} \|h\|^3_{\st{\infty}{6}([T,\widetilde{T})\times \R^3)},\end{align*}
Consequently
\begin{equation}\label{N-12bound}
\|N(\phi, h)\|_{\st{1}{2}([T,\widetilde{T})\times \R^3)}\leq C \|h\|^2_{\st{2}{\infty}([T,\widetilde{T})\times \R^3)},
\end{equation}
 and 
 \begin{align}\|N_{\rho_i}\|_{L^{\infty}_t([\widetilde{T},\infty))} & \leq C \|\rho_i\|_{L^{6}_x(\R^3)}\|N(\vp,h)\|_{\st{\infty}{\frac{6}{5}}([\widetilde{T},\infty)\times\R^3)}\notag \\ & \leq C \sum_{i=2}^5\|\vp\|_{L^6_x}^{5-i}\|h\|_{\st{\infty}{6}( [\widetilde{T},\infty)\times \R^3)}^i \leq C\epsilon^2. \label{N-inftybound}\end{align}
 Using (\ref{h-small}) and $$\|h\|_{\st{2}{\infty}( [T,\widetilde{T})\times \R^3)}\lesssim \|(\lambda_1, \cdots, \lambda_n, \gamma)\|_{X([T,\widetilde{T}))}$$ with constant depending on  $\|\rho_i\|_{L^\infty_x}$, we can combine estimates (\ref{lambda}),(\ref{gamma}) with (\ref{small-lambda}), (\ref{small-gamma}) and (\ref{N-bound})-(\ref{N-inftybound}) to get
 \begin{align*}
  \|(\lambda_1, \cdots, \lambda_n, \gamma)\|_{X([T,\widetilde{T}))} \leq K\left\{\epsilon  +   \|(\lambda_1, \cdots, \lambda_n, \gamma)\|_{X([T,\widetilde{T}))}^2 +  \epsilon^2 \right\},  \end{align*}
Here $K$ is some constant depending only on the constants in Strichartz inequalities for equation (\ref{eq:linearequationwithpotential}) and  $\|\vp\|_{L^6_x}$ and $\|\rho_i\|_{L^\infty_x}$.  Since this estimate is true for all $\widetilde{T} > T$,  we can choose $\epsilon\ll 1$, which can be achieved by taking $T$ sufficiently large,  such that 
  \[ (4K^2+1)\epsilon <1.\]
  By a continuity argument  we then obtain that 
 \[  \|(\lambda_1, \cdots, \lambda_n, \gamma)\|_{X([T,\infty))}  \leq 2K\epsilon.\]
 which implies that $\|h\|_{\st{2}{\infty}( [T,\infty)\times \R^3)}\lesssim \epsilon$. Using interpolation between the $\st{\infty}{6}$ and $\st{2}{\infty}$ norms, we can also obtain 
\begin{equation} \|h\|_{L^q_tL^r_x( [T,\infty)\times \R^3)}  \lesssim  \epsilon , \quad \text{ for any admissible pair } (q,r), q\geq 2.\label{small-h}\end{equation}
In particular, we infer that $h\in L^5_tL^{10}_x([0,\infty)\times \R^3)$. 

\smallskip

\noindent
{\it Step 2: construction of the center-stable manifold near a solution $U$.}  Given a radial finite energy solution $U$ to (\ref{eq:mainequation}) satisfying (\ref{scatter-vp}), 
we consider another radial finite energy solution $u$, with $\|\OR{U}(T)-\OR{u}(T)\|_{\HL(\R^3)}$ small for a fixed large time $T$ from Step 1 (we may need to take $T$ large to close the estimates below, which of course can be done). We write $u= U + \eta $, then  
$\eta$ satisfies the equation 
\[\eta_{tt}-\Delta \eta - V(x)\eta +(U+\eta)^5-U^5=0, \,\,(t,x)\in(T,\infty).\]
Plugging in $U=\vp +h$, we can  further write the equation as 
\begin{equation}\label{eta-eq}
 \eta_{tt} + \mathcal{L}_{\vp}\eta +\tilde{{N}}(\vp, h, \eta)=0, \,\,(t,x)\in(T,\infty),
\end{equation}
with $$\tilde{{N}}(\vp, h, \eta) = (\vp +h +\eta)^5- (\vp + h)^5 - 5\vp^4\eta$$
We note that $\tilde{N}$ still contains terms linear in $\eta$. However, a closer inspection shows that the coefficients of the linear terms in $\eta$ decay in both space and time, and can be made small if we choose $T$ sufficiently large.  First write $$\eta =\tl_1(t)\rho_1 +\cdots+ \tl_n(t)\rho_n +\tg, \qquad \tg\perp \rho_i$$ for $i=1,\cdots, n$. We use similar arguments as in step~1 to obtain a solution $\eta$ which stays small for all positive times with given $(\tl_1(T),\cdots,\tl_n(T))$ and $(\tg,\dot{\tg})(T)$.
  We can obtain equations for $\tl_i, \tg$ similar to (\ref{system}). Since we seek  a forward solution which grows at most polynomially, we obtain a similar  necessary and sufficient stability condition as (\ref{condition}) 
\begin{equation}\label{condition2} \dot{\tl}_i(T)= - k_i\tl_i(T)- \int_T^\infty e^{k_i(T-s)}\tilde{N}_{\rho_i}(s)ds.
\end{equation}
Using equations (\ref{eta-eq}) and (\ref{condition2}) we arrive at the system of equations for $\tl_i$ and $ \tg$,
\begin{equation}\label{system-2}\left\{\begin{aligned}
\tl_i(t)&=e^{-k_i(t-T)}\left[\tl_i(T)+\frac{1}{2k_i} \int_T^\infty e^{k_i (T-s)}\tilde{N}_{\rho_i}(s)ds\right]-\frac{1}{2k_i}\int_T^\infty e^{-k_i|t-s|}\tilde{N}_{\rho_i}(s)\,ds, \\
\tg(t)&=\cos (\omega (t-T))\tg(T)+\frac{1}{\omega}\sin (\omega (t-T))\dot{\tg}(T)+\frac{1}{\omega}\int_T^t\sin (\omega(t-s))\tilde{N}_c (s)\,ds. 
\end{aligned}\right.\end{equation}
Define 
\begin{equation} \|(\tl_1,\ldots, \tl_n, \tg)\|_{X}:= \sum_{i=1}^n\|\tl_i(t)\|_{L^\infty_t \cap L_t^2([T,  \infty))} + \|\tg\|_{L^\infty_t{\dot{H}^1}\cap \st{2}{\infty}( [T,\infty)\times \R^3)}.\label{define-X}\end{equation}
Estimating system (\ref{system-2}), we obtain that 
 \begin{align}
\|\tl_i(t)\|_{L^{\infty}\cap L^2([T, \infty))}\lesssim & |\tl_i(T)| + \|\tilde{N}_{\rho_i}\|_{L^1_t([T,\infty))}  \lesssim |\tl_i(T)| +  \|\tilde{N}\|_{\st{1}{2}( [T,\infty)\times \R^3)}, \label{tilde-lambda} \\
\|\tg\|_{L^\infty_t{\dot{H}^1}\cap\st{2}{\infty}( [T,\infty)\times\R^3)}\lesssim& \|(\tg(T), \dot{\tg}(T))\|_{\hl} + \|\tilde{N}\|_{\st{1}{2}( [T,\infty)\times \R^3)}. \label{tilde-gamma}
\end{align}
Recalling $|\tilde{N}| \lesssim  \sum_{j=1}^4 |\vp^{4-j} h^j \eta |+ \sum_{k\geq 2, i+j+k=5} |\vp^i h^j   \eta^k|$, we have 
\begin{align*}
\|\vp^3 h\eta\|_{\st{1}{2}( [T,\infty)\times \R^3)}& \leq  \|\vp\|_{L^6_x}^3 \|h\|_{\st{2}{\infty}([T,\infty)\times \R^3)}\|\eta\|_{\st{2}{\infty}([T,\infty)\times \R^3)},\\
\|\vp^2 h^2\eta\|_{\st{1}{2}( [T,\infty)\times \R^3)}& \leq  \|\vp\|_{L^6_x}^2 \|h\|_{\st{\infty}{6}([T,\infty)\times \R^3)}\|h\|_{\st{2}{\infty}([T,\infty)\times \R^3)}\|\eta\|_{\st{2}{\infty}([T,\infty)\times \R^3)},\\
\|\vp h^3\eta\|_{\st{1}{2}( [T,\infty)\times \R^3)}& \leq  \|\vp\|_{L^6_x}\|h\|^2_{\st{\infty}{6}([T,\infty)\times \R^3)}\|h\|_{\st{2}{\infty}([T,\infty)\times \R^3)}\|\eta\|_{\st{2}{\infty}([T,\infty)\times \R^3)},\\
\| h^4\eta\|_{\st{1}{2}( [T,\infty)\times \R^3)}& \leq   \|h\|^3_{\st{\infty}{6}([T,\infty)\times \R^3)}\|h\|_{\st{2}{\infty}([T,\infty)\times \R^3)}\|\eta\|_{\st{2}{\infty}([T,\infty)\times \R^3)}.
\end{align*}
Using (\ref{small-h}), we get that 
\begin{equation}\label{N1-estimate}\left\|\sum_{j=1}^4 \vp^{4-j} h^j \eta\right\|_{\st{1}{2}( [T,\infty)\times \R^3)}\lesssim \epsilon  \|\eta\|_{ \st{2}{\infty}([T,\infty)\times \R^3)}.\end{equation}
The higher order terms are easier to estimate. We can always place $h$ in $L^\infty_tL^6_x$, whence 
\begin{align}\label{N2-estimate}
\left\| \sum_{k\geq 2, i+j+k=5}\vp^i h^j   \eta^k\right\|_{\st{1}{2}( [T,\infty)\times \R^3)}\lesssim \sum_{k=2}^5\|\eta\|^k_{L^\infty_t{\dot{H}^1}\cap \st{2}{\infty}([T,\infty)\times \R^3)}.
\end{align}

Since $\|\eta\|_{L^{\infty}_t\dot{H}^1\cap\st{2}{\infty}([T,\infty)\times \R^3)}\lesssim \|(\tl_1, \cdots, \tl_n, \tg)\|_{X([T,\infty))}$, we can combine the preceding estimates and get that
\begin{align*}\|(\tl_1,\cdots, \tl_n, \tg)\|_{X([T,\infty))}\leq & \,  K \left(\sum_{i=1}^n |\tl_i(T)| + \|(\tg(T), \dot{\tg}(T))\|_{\hl} \right)
\\ &+ K\epsilon \|(\tl_1,\cdots, \tl_n, \tg)\|_{X([T,\infty))} + K \sum_{k=2}^5\|(\tl_1,\cdots, \tl_n, \tg)\|_{X([T,\infty))}^k,\end{align*}
where $K>1$ is a constant only depending on the constants in the Strichartz estimates for equation (\ref{eq:linearequationwithpotential}) and $\|\phi\|_{L^6(\R^3)}$ and $\|\rho_i\|_{L^\infty_x}$. This inequality implies
that  if we take $\epsilon=\epsilon_0$ sufficiently small (which can be achieved by choosing $T$ suitably large), and $\delta<\epsilon_0$ with 
\begin{equation}\label{small-data}\sum_{i=1}^n |\tl_i(T)| + \|(\tg(T), \dot{\tg}(T))\|_{\hl}\leq \delta,\end{equation}
such that $K\epsilon_0 <\frac{1}{4}$ and $K^2\delta<\frac{1}{32}$, then 
the map defined by the right-hand side of system (\ref{system-2}) takes a ball  $B_{2K\delta}(0)$ into itself. Moreover, we can check by the same argument   that this map is in fact
 a contraction. Thus for any given small $(\tl_1(T), \cdots \tl_n(T), \tg(T))$ satisfying (\ref{small-data}), we obtain a unique fixed point of (\ref{system-2}). Then 
\begin{equation*}
u(t,x):=U(t,x) +\sum_{i=1}^k \tl_i(t)\rho_i+\tg(t,x)
\end{equation*}
 solves the equation (\ref{eq:mainequation}) on $\R^3\times [T,\infty)$, satisfying 
\begin{equation}\|\OR{u}-\OR{U}\|_{L_t^\infty( [T,\infty);\hl)}
\leq C \delta \label{epsilon-difference}\end{equation}
 with Lipschitz dependence on the data $\tl_i(T)$ and $(\tg(T),\dot{\tg}(T))$. Since the nonlinearity $\tilde{N}$ only involves integer powers of $\eta$, we see that the integral terms in (\ref{system-2}) have smooth dependence on $\tl_i, \tg$. Hence we conclude that $\tl_i(t),\tg(t,x)$ and the solution $u(t,x)$ actually have smooth dependence on the data.

 By the estimates on $\tilde{\lambda}_i$ and $\tilde{\gamma}$, we conclude in addition that
\begin{equation*}
\eta=\sum_{i=1}^n \tl_i(t)\rho_i+\tg(t,x)\in L^2_tL^{\infty}_x([T,\infty)\times \R^3),
\end{equation*}
 hence $\OR{u}(t)$ scatters to the same steady state as $\OR{U}(t)$ which is $(\phi,0)$.

We can now define
\begin{equation}
\Psi: \overrightarrow{U}(T)+\left(B_{\epsilon_0}((0,0))\cap X_s\right)\longrightarrow \dot{H}^1\times L^2,
\end{equation}
as follows: for any $(\tg_0,\tg_1)\in P^{\perp}\left(\RHL(\R^3)\right)$   and $\tl_i\in \R$ such that
\begin{equation*}
\mathcal{\xi}:=\sum_{i=1}^n\tl_i(\rho_i,-k_i\rho_i)+(\tg_0,\tg_1)+ \overrightarrow{U}(T)\in  \overrightarrow{U}(T)+\left(B_{\epsilon_0}((0,0))\cap X_s\right),
\end{equation*}
set 
\begin{equation*}
\tl_i(T)=\tl_i, \,\,\,{\rm for\,\,}i=1,\dots,n \,\,\,{\rm and}\,\,(\tg(T),\dot{\tg}(T))=(\tg_0,\tg_1).
\end{equation*}
Then with $\dot{\tl}_i(T)$ given by  (\ref{condition2}), we define
\begin{equation*}
\Psi(\mathcal{\xi}):=\left(\sum_{i=1}^n\tl_i(T)\rho_i+\tg_0,\sum_{i=1}^n\dot{\tl}_i(T)\rho_i+\tg_1\right) +\OR{U}(T).
\end{equation*}
If $\epsilon_0$ is chosen sufficiently small, then $\dot{\tl}_i$ is uniquely determined by contraction mapping in the above.  
We define  $\widetilde{\mathcal{M}}$ as the graph of $\Psi$ and let $\mathcal{M}$ be $\OR{S}(-T)(\widetilde{\mathcal{M}})$.
We can then check that $\Psi,\,\mathcal{M},\,\widetilde{\mathcal{M}}$ verify the requirements of the theorem. Since $\OR{S}(T)$ is a diffeomorphism, $\mathcal{M}$ is a $C^1$ manifold. 
 We remark that due to the presence of radiations $\OR{U}^L$, the graph $\widetilde{\mathcal{M}}$ is in general not tangent to the center-stable subspace $X_s$.


\smallskip

\noindent 
{\it Step 3: unconditional uniqueness.} Now suppose that we are given a solution $u$ to equation (\ref{eq:mainequation}), which satisfies
 \[\|\OR{u}-\OR{U}\|_{L^\infty( [0,\infty);\hl)}\leq\epsilon_1\ll\epsilon_0.\]
 We need to show that $\OR{u}(T)\in \widetilde{\mathcal{M}}$. We denote \[\eta(t,x) = u(t,x)-U(t,x) = \sum_{i=1}^n \tl_i(t)\rho_i +\tg(t,x),\] 
 then  $\OR{\eta}\in L^\infty_t([0,\infty);\hl)$. 
By the fact that $u, U$ are solutions to equation (\ref{eq:mainequation}) and Strichartz estimates, we see that 
 $\eta \in L^q_tL^r_x (I\times \R^3)$ for any finite interval $I\subseteq[0,\infty)$ and admissible pair $(q,r)$, we get for any $\tilde{T}>T$, 
 \begin{align} 
& \|\tl_i(t)\|_{L_t^\infty([T,\infty))} +  \|\vec{\tg}(t,x)\|_{ L_t^\infty( [T,\infty);\hl) }\lesssim \epsilon_1,\label{delta-small}\\
& \tl_i(t)\in    L^2([T,\tilde{T})), \notag \\
&  \tg(t,x)\in L^2_tL^{\infty}_x( [T,\tilde{T})\times \R^3). \notag
 \end{align}
 Notice the $L^\infty$ bound on $\tl_i$ implies that the stability condition   (\ref{condition2}) must hold true, so we are again reduced to considering system (\ref{system-2}). 
 Now we wish to show that $\tl_i(t)\in L^2([T,\infty)) $ and $\tg(t,x)\in L^2_tL^{\infty}_x([T,\infty)\times \R^3)$. To do this, we follow similar arguments as in step 1. Define the norm
\begin{align*} \|(\tl_1,\cdots, \tl_n, \tg)\|_{X([T,\tilde{T}))}:&= \sum_{i=1}^n\|\tl_i(t)\|_{L_t^2([T, \tilde{T}))} + \|\tg\|_{\st{2}{\infty}([T,\tilde{T})\times \R^3)}
\end{align*}
By estimating (\ref{system-2}) similar to (\ref{lambda}), (\ref{gamma}), we get 
\begin{align*}
\sum_{i=1}^n\|\tl_i(t)\|_{L^2([T, \tilde{T}))} + \|\tg\|_{\st{2}{\infty}( [T,\tilde{T})\times \R^3)}\lesssim &\sum_{i=1}^n |\tl_i(T)| + \|(\tg(T), \dot{\tg}(T))\|_{\hl}\\ &  + \|\tilde{N}\|_{\st{1}{2}( [T,\tilde{T})\times \R^3)} + \|\tilde{N}\|_{L^{\infty}_tL^{\frac{6}{5}}_x([\tilde{T},\infty)\times \R^3)}  \end{align*}
Recall we have $h=U-\phi, \|h\|_{\st{2}{\infty}( [T,\infty)\times \R^3)} \leq \epsilon=\epsilon_0$.  Using the same estimate as in (\ref{N1-estimate}) and (\ref{N2-estimate}) on the time interval $[T,\tilde{T})$, we obtain that
\[\|\tilde{N}\|_{\st{1}{2}( [T,\tilde{T})\times \R^3)} \lesssim \epsilon_0 \|\eta\|_{\st{2}{\infty}( [T,\tilde{T})\times \R^3)}+\sum_{k=2}^5\|\eta\|^k_{L^\infty_t{\dot{H}^1}\cap \st{2}{\infty}([T,\tilde{T})\times \R^3)},\]
and 
\[\|\tilde{N}\|_{L^{\infty}_tL^{\frac{6}{5}}_x( [\tilde{T},\infty)\times \R^3)}   \lesssim \sum_{i+j+k=5, k\geq 1}\|\vp\|_{L^6_x}^i\|h\|^j_{L^\infty_t L^6_x( [\tilde{T},\infty)\times \R^3)}\|\eta\|^k_{L^\infty_tL^6_x( [\tilde{T},\infty)\times \R^3)}\lesssim \epsilon_1.\]
Hence  
\[\|(\tl_1,\cdots, \tl_n, \tg)\|_{X([T,\tilde{T}))} \lesssim \epsilon_1 + \epsilon_0 \|(\tl_1,\cdots, \tl_n, \tg)\|_{X([T,\tilde{T}))}+K \sum_{k=2}^5\|(\tl_1,\cdots, \tl_n, \tg)\|_{X[T,\tilde{T})}^k.\]
From this, by a continuity argument, we can conclude that
\[\|(\tl_1,\cdots, \tl_n, \tg)\|_{X([T,\infty))} \leq \liminf_{\tilde{T}\rightarrow \infty}\|(\tl_1,\cdots, \tl_n, \tg)\|_{X([T,\tilde{T}))}\leq C \epsilon_1 <\epsilon_0,\]
and the contraction mapping theorem then implies $\OR{u}(T)\in\widetilde{\mathcal{M}}$.  

\smallskip 

\noindent
{\it Step 4: summary.} Let us sum up our construction as follows: consider any point $(U_0,U_1)\in \mathcal{M}_{\vp}$, which generates a solution $U(t,x)$ to equation (\ref{eq:mainequation}) satisfying (\ref{scatter-vp}). For sufficiently large time $T$, we can construct a smooth graph $\widetilde{\mathcal{M}}$ of co-dimension $n$ in $$B_{\epsilon_0}(\OR{U}(T))\in \RHL$$ such that solutions starting from $\widetilde{\mathcal{M}}$ remain close to $\OR{U}(t)$ for all $t\ge T$ and scatter to $(\phi,0)$. The graph can also be parametrized smoothly by $$\tl_1(T), \cdots \tl_n(T)\in \R, \vec{\tg}(T)\in P^{\perp}(\RHL)$$ in the following sense.  For the parameters satisfying 
\[\sum_{i=1}^n |\tl_i(T)| + \|(\tg(T), \dot{\tg}(T))\|_{\hl}\leq \epsilon_0\]
there exists a unique solution ${u}$ to equation (\ref{eq:mainequation}) on $t\geq T$ satisfying 
\[u(T) = U(T) +\sum_{i=1}^n \tl_i(T)\rho_i +\tg(T,x), \quad P^{\perp}\dot{u}(T) = \dot{\tg}(T)+P^{\perp}\partial_tU(T),\]
with the property that $u(t)$ scatters to $\phi$,     
and \[\|\OR{u}(t) -\OR{U}(t)\|_{L^\infty_t ([T,\infty),\hl)}\leq C \epsilon_0.\]
Moreover, any solution that satisfies $\|\OR{u}(t)-\OR{U}(t)\|_{\HL}<\epsilon_1$ with some $\epsilon_1<\epsilon_0$, for all times $t\geq T$ necessarily starts on $\widetilde{\mathcal{M}}$. Using the solution flow $\OR{S}(t)$, we pull back our construction to time $0$, $\mathcal{M}=\OR{S}(-T)\widetilde{\mathcal{M}}$, and the theorem is proved.
\end{proof}
Now we give a proof of Claim~\ref{Claim}.  Claim~\ref{Claim} will be proved as a consequence of the following lemma. 
\begin{lemma}
Let $\overrightarrow{U}^L$ be a radial finite energy free radiation and $(\phi, 0)$ be a steady state to equation (\ref{eq:mainequation}). Recall that 
\[\omega =\sqrt{P^{\perp}(-\Delta - V +5\phi^4)}.\]
Let $\gamma$ be the solution to 
\begin{equation}
\left\{\begin{aligned} \partial_{tt}\gamma +\omega^2 \gamma &=0, \hspace{1cm} \text{ in } [T,\infty)\times \R^3,
\\
\overrightarrow{\gamma}(T) & = P^{\perp} (\overrightarrow{U}^L(T)). 
\end{aligned}\right. 
\end{equation}
For any $\epsilon>0$, if we take $T=T(\epsilon, \overrightarrow{U}^L)>0$ sufficiently large, then 
\begin{equation}
 \|\gamma\|_{\st{2}{\infty}([T,\infty)\times \R^3)} <\epsilon. 
\end{equation}
\end{lemma}
\begin{proof} For a given $\epsilon>0$, fix $0<\delta \ll \epsilon$ to be determined below. We can take a radial smooth compactly supported (in space) free radiation $\overrightarrow{\widetilde{U}}^L$ such that 
\begin{equation}
\|\overrightarrow{U}^L(0)-\overrightarrow{\widetilde{U}}^L(0)\|_{\hl(\R^3)} \leq \delta.  \label{UU-small}
\end{equation}
Let us assume that ${\rm supp}\overrightarrow{\widetilde{U}}^L(0) \Subset B_R(0)$ for some $R>0.$ Hence by strong Huygens' principle, for large time $T$ we have ${\rm supp}\overrightarrow{\widetilde{U}}^L(T) \Subset B_{T+R}\backslash B_{T-R}$. Since $\overrightarrow{\widetilde{U}}^L$ is a free radiation, we see that 
\begin{equation}
\partial_{tt} {\widetilde{U}}^L -\Delta  {\widetilde{U}}^L - V  {\widetilde{U}}^L + 5\phi^4  {\widetilde{U}}^L = - V  {\widetilde{U}}^L +5\phi^4  {\widetilde{U}}^L, \quad\text{ in } (0,\infty)\times \R^3. \label{eq:tildeU}
 \end{equation}
 By the decay property of $V, \, 5\phi^4$ and the support property of $ {\widetilde{U}}^L$, simple calculations show that 
 \[\lim_{T\rightarrow \infty } \|- V {\widetilde{U}}^L +5\phi^4 {\widetilde{U}}^L\|_{\st{1}{2}([T,\infty)\times \R^3)} =0.\]
 Choose $T$ sufficiently large, such that 
 \begin{equation}
  \|- V  {\widetilde{U}}^L +5\phi^4  {\widetilde{U}}^L\|_{\st{1}{2}([T,\infty)\times \R^3)} \leq \delta.\label{RHS-small}
 \end{equation}
 Note that $\overrightarrow{v}:= \overrightarrow{\gamma} - P^{\perp}\overrightarrow{\widetilde{U}}^L$ solves 
 \[\partial_{tt} v +\omega^2 v = - P^{\perp}\left(  - V  {\widetilde{U}}^L +5\phi^4  {\widetilde{U}}^L \right), \quad (t,x)\in [T, \infty)\times \R^3, \]
 with initial data $\overrightarrow{v}(T)= P^{\perp}\left(
 \overrightarrow{U}^L(T) - P^{\perp}\overrightarrow{\widetilde{U}}^L(T)\right)$. By the bounds (\ref{UU-small}) and (\ref{RHS-small}), energy conservation for free radiation, and Strichartz estimates from Lemma \ref{lm:strichartzwithpotential}, we can conclude that 
 \begin{equation}\|v\|_{\st{2}{\infty}([T,\infty)\times \R^3)}\leq C\delta. \label{small-v}\end{equation}
 Since $\overrightarrow{\widetilde{U}}^L$ is a finite energy free radiation, if we choose $T$ sufficiently large, we have 
 \begin{equation}
 \| {\widetilde{U}}^L\|_{\st{2}{\infty}([T,\infty)\times \R^3)}\leq C\delta. \label{small-UT}
 \end{equation}
 Combining bounds (\ref{small-v}) and (\ref{small-UT}), and fixing $\delta$ small, the lemma is proved. 
\end{proof}
Now the proof of Claim~\ref{Claim} is easy. Note that due to the fact that 
\[\lim_{T\rightarrow \infty} \|\overrightarrow{U}(T) - (\phi,0) -\overrightarrow{U}^L(T)\|_{\hl(\R^3)} =0,\]
we see that the initial data for $\gamma$ satisfies 
\[\lim_{T\rightarrow \infty} \|\overrightarrow{\gamma}(T)-P^{\perp}\overrightarrow{U}^L(T)\|_{\hl(\R^3)} =0.\]
Hence the claim follows from the above lemma and Strichartz estimates. 

\section{Profile decomposition and channel of energy inequality}
\label{sec:4}

In this section we recall some well-known properties of profile decompositions first introduced in the context of wave equations by Bahouri, Gerard\cite{BaGe}, and channel of energy inequalities discovered by Duyckaerts, Kenig and Merle \cite{DKM,DKM1}. For both, we require the versions adapted to the wave equation with a potential. We refer the reader to \cite{JiaLiuXu} for proofs. We first recall the following perturbation result.
\begin{lemma}\label{lm:perturbationresult}
Let $0\in I\subset \R$ be an interval of time. Suppose $\tilde{u}(t,x)\in C_t(I,\dot{H}^1(\R^3))$ with $\|\tilde{u}\|_{L^5_tL^{10}_x(I\times \R^3)}\leq M<\infty$, $\|a\|_{L^{5/4}_tL^{5/2}_x(I\times \R^3)}\leq \beta<\infty$ and $e(t,x),\,f(t,x)\in L^1_tL^2_x(I\times \R^3)$, satisfy
\begin{equation}
\partial_{tt}\tilde{u}-\Delta \tilde{u}+a(t,x)\tilde{u}+\tilde{u}^5=e,
\end{equation}
with initial data $\overrightarrow{\tilde{u}}(0)=(\tilde{u}_0,\tilde{u}_1)\in \dot{H}^1\times L^2$. Suppose for some sufficiently small positive $\epsilon<\epsilon_0=\epsilon_0(M,\beta)$, 
\begin{equation}
\||e|+|f|\|_{L^1_tL^2_x(I\times \R^3)}+\|(u_0,u_1)-(\tilde{u}_0,\tilde{u}_1)\|_{\dot{H}^1\times L^2}<\epsilon.
\end{equation}
Then there is a unique solution $u\in C(I,\dot{H}^1)$ with $\|u\|_{L^5_tL^{10}_x(I\times \R^3)}<\infty$, satisfying the equation
\begin{equation}
\partial_{tt}u-\Delta u+a(t,x)u+u^5=f,
\end{equation}
with initial data $\OR{u}(0)=\overrightarrow{u}(0)=(u_0,u_1)$. Moreover, we have the following estimate
\begin{equation}
\sup_{t\in I}\|\overrightarrow{u}(t)-\overrightarrow{\tilde{u}}(t)\|_{\dot{H}^1\times L^2}+\|u-\tilde{u}\|_{L^5_tL^{10}_x(I\times \R^3)}<C(M,\beta)\epsilon.
\end{equation}
\end{lemma}

Lemma \ref{lm:perturbationresult} has the following implication concerning  global existence and scattering for defocusing energy critical wave equation with  potential,  decaying both in space and time.

\begin{lemma}\label{lm:globalregularityofperturbedwave}
Let $I$ be an interval of time and $a\in L^{5/4}_tL^{5/2}_x\cap L^1_tL^3_x(I\times \R^3)$, and $f\in L^1_tL^2_x(I\times \R^3)$, with bounds $\|a\|_{L^{5/4}_tL^{5/2}_x}+\|a\|_{L^1_tL^3_x}\leq M$ and $\|f\|_{L^1_tL^2_x}\leq \beta$. Then there exists a unique solution $u\in C(I,\dot{H}^1)\cap L^5_tL^{10}_x(I\times \R^3)$ to the equation
\begin{equation}\label{eq:perturbedwave}
\partial_{tt}u-\Delta u+a(t,x)u+u^5=f,
\end{equation} 
with initial data $(u_0,u_1)\in \dot{H}^1\times L^2$ ($\|(u_0,u_1)\|_{\dot{H}^1\times L^2}\leq E$). Moreover, we have
\begin{equation}\label{eq:bound}
\|u\|_{L^5_tL^{10}_x(I\times \R^3)}\leq C(E,M,\beta).
\end{equation}
Thus if $I=\R$, then there exist solutions $u^L_{+},\,u^L_{-}$ to free wave equation, such that
\begin{eqnarray}
&&\lim_{t\to +\infty}\|u(t)-u^L_{+}(t)\|_{\dot{H}^1\times L^2}=0,\\
&&\lim_{t\to -\infty}\|u(t)-u^L_{-}(t)\|_{\dot{H}^1\times L^2}=0.
\end{eqnarray}
\end{lemma}

The following lemma shows that for potentials $a\in L^{5/4}_tL^{5/2}_x$ (thus with space-time decay), large or small profiles are essentially not influenced by the potential.

\begin{lemma}\label{lm:aux}
Let $a\in L^{5/4}_tL^{5/2}_x(\R\times \R^3)$ and $U^L$ be a solution to the free wave equation in $\R\times \R^3$. Take parameters $(\lambda_n,t_n)$ with $\lambda_n>0,\,t_n\in \R$. Assume one of the following conditions holds:\\
1. $t_n\equiv 0$, $\lim\limits_{n\to\infty}(\lambda_n+\frac{1}{\lambda_n})=\infty$,\\
2. $\lim\limits_{n\to\infty}\frac{t_n}{\lambda_n}\in\{\pm\infty\}$.\\
Let $U$ be the nonlinear profile associated with $U^L,\,\lambda_n,\,t_n$. More precisely
\begin{equation}\label{eq:defocusingwave}
\partial_{tt}U-\Delta U+U^5=0\,\,{\rm in\,} \R\times \R^3,
\end{equation}
with $\overrightarrow{U}(0)=(U^L(0),\partial_tU^L(0))$\, if\, $t_n\equiv 0$; or with 
\begin{equation}
\lim_{t\to +\infty}\|\overrightarrow{U}(t)-\overrightarrow{U^L}(t)\|_{\dot{H}^1\times L^2}=0,\,\,(\lim_{t\to-\infty})
\end{equation}
 if\, $\lim_{n\to\infty}\frac{t_n}{\lambda_n}=-\infty$ ($\lim=\infty$ respectively). Let $u_n$ be the solution to the Cauchy problem
\begin{equation}
\partial_{tt}u_n-\Delta u_n+a(t,x)u_n+u_n^5=0\,\,{\rm in\,}\R\times \R^3,
\end{equation}
with $\overrightarrow{u}_n(0)=\left(\frac{1}{\lambda_n^{1/2}}U^L(-\frac{t_n}{\lambda_n},\frac{x}{\lambda_n}),\frac{1}{\lambda_n^{3/2}}\partial_tU^L(-\frac{t_n}{\lambda_n},\frac{x}{\lambda_n})\right)$. Then
\begin{equation}
\lim_{n\to\infty}\left(\sup_{t\in \R}\|\overrightarrow{u_n}(t)-\overrightarrow{U_n}(t)\|_{\dot{H}^1\times L^2}+\|u_n-U_n\|_{L^5_tL^{10}_x(\R\times \R^3)}\right)=0,
\end{equation}
where $\overrightarrow{U}_n(x,t)=\left(\frac{1}{\lambda_n^{1/2}}U(\frac{t-t_n}{\lambda_n},\frac{x}{\lambda_n}),\frac{1}{\lambda_n^{3/2}}\partial_tU(\frac{t-t_n}{\lambda_n},\frac{x}{\lambda_n})\right)$.
\end{lemma}

The following profile decomposition adapted for wave equation with potential plays an important role in our analysis.

\begin{lemma}\label{lm:profilewithpotential}
Let $a\in L^{5/4}_tL^{5/2}_x\cap L^{1}_tL^3_x(\R\times \R^3)$. Suppose a radial sequence $(u_{0n},u_{1n})\in \dot{H}^1\times L^2$ is uniformly bounded and that we have the following linear profile decompositions (see Bahouri-Gerard\cite{BaGe})
\begin{equation}
(u_{0n},u_{1n})=\overrightarrow{U}^L_1(0)+\sum_{j=2}^J(\frac{1}{\lambda_{jn}^{1/2}}U^L_j(-\frac{t_{jn}}{\lambda_{jn}},\frac{x}{\lambda_{jn}}),\frac{1}{\lambda_{jn}^{3/2}}\partial_tU^L_j(-\frac{t_{jn}}{\lambda_{jn}},\frac{x}{\lambda_{jn}}))+\overrightarrow{w}_{Jn}(0),
\end{equation}
with the following properties:\\
\begin{eqnarray*}
&&U^L_j,\,w_{Jn} \,{\rm \,are\,\,radial\,\,and\, \,solve\,\,the\,\, free\, \,wave\,\, equation}\,\,{\rm for\,\,each}\,\,j,\,J;\\
&&{\rm either}\,\,t_{jn}\in \R,\,\lambda_{jn}>0,\,\lim_{n\to\infty}\frac{t_{jn}}{\lambda_{jn}}\in\{\pm\infty\}{\rm \,\,or\,\,}t_{jn}\equiv 0,\,\lim_{n\to\infty}\left(\lambda_{jn}+\frac{1}{\lambda_{jn}}\right)=\infty;\\
&& {\rm for\,}\,j\neq j',\,\lim_{n\to\infty}\left(\frac{\lambda_{jn}}{\lambda_{j'n}}+\frac{\lambda_{j'n}}{\lambda_{jn}}+\frac{|t_{jn}-t_{j'n}|}{\lambda_{jn}}\right)=\infty;\\
&&{write}\,\,w_{Jn}(t,x)=\frac{1}{\lambda_{jn}^{1/2}}\tilde{w}^j_{Jn}(\frac{t-t_{jn}}{\lambda_{jn}},\frac{x}{\lambda_{jn}}),\,\,{then}\,\,\tilde{w}_{Jn}^j\rightharpoonup 0,\,\,{\rm and\,}\,w_{Jn}\rightharpoonup 0,\,\,{\rm as}\,\,n\to\infty;\\
&&\lim_{J\to\infty}\limsup_{n\to\infty}\|w_{Jn}\|_{L^5_tL^{10}_x(\R \times \R^3)}=0.
\end{eqnarray*}
Let $U_1$ satisfy
\begin{equation}
\partial_{tt}U_1-\Delta U_1+a(t,x)U_1+U_1^5=0,{\rm \,in\,}\R\times \R^3,
\end{equation}
with $\overrightarrow{U}_1(0)=\overrightarrow{U}^L_1(0)$.
Let $U_j$ be the nonlinear profile associated to $U^L_j,\,\lambda_{jn},\,t_{jn}$ as defined in Lemma \ref{lm:aux} for $j\ge 2$. Let $u_n$ be the solution to
\begin{equation}
\partial_{tt}u_n-\Delta u_n+a(t,x)u_n+u_n^5=0,\,{\rm in\,}\R\times \R^3,
\end{equation}
with $\overrightarrow{u}_n(0)=(u_{0n},u_{1n})$. Then we have the following decomposition: 
\begin{equation}\label{eq:decomposition*}
\overrightarrow{u}_n(t)=\overrightarrow{U_1}(t)+\sum_{j=2}^J \overrightarrow{U}_{jn}(t)+\overrightarrow{w}_{Jn}(t)+\overrightarrow{r}_{Jn}(t),
\end{equation}
with 
\begin{equation}\label{eq:errorterm*}
\lim_{J\to\infty}\limsup_{n\to\infty}\left(\sup_{t\in \R}\|\overrightarrow{r}_{Jn}(t)\|_{\dot{H}^1\times L^2}+\|r_{Jn}\|_{L^5_tL^{10}_x(\R\times \R^3)}\right)=0,
\end{equation}
where $\overrightarrow{U}_{jn}(t,x)=\left(\frac{1}{\lambda_{jn}^{1/2}}U_j(\frac{t-t_{jn}}{\lambda_{jn}},\frac{x}{\lambda_{jn}}),\frac{1}{\lambda_{jn}^{3/2}}\partial_tU_j(\frac{t-t_{jn}}{\lambda_{jn}},\frac{x}{\lambda_{jn}})\right)$. Moreover, denoting $U_{1n}=U_1$, for $\rho_n>\sigma_n>0$ and $\theta_n\in \R$ we have the following orthogonality property for $1\leq j\not= j'$
\begin{eqnarray}
&&\label{eq:no1}\lim_{n\to\infty}\int_{\sigma_n<|x|<\rho_n}\nabla U_{jn}\nabla U_{j'n}+\partial_tU_{jn}\partial_tU_{j'n}(\theta_n,x)\,dx=0;\\
&&\label{eq:no2}\lim_{n\to\infty}\int_{\sigma_n<|x|<\rho_n}\nabla U_{jn}\nabla w_{Jn}+\partial_tU_{jn}\partial_tw_{Jn}(\theta_n,x)\,dx=0.
\end{eqnarray}
\end{lemma}

We also need the following channel of energy inequality from \cite{JiaLiuXu}, which was proved with similar arguments as in Duyckaerts, Kenig and Merle \cite{DKM}.
\begin{theorem}\label{th:channelofenergy}
Suppose radial finite energy $(u_0,u_1)\not\equiv (\phi,0)$ for any steady state solution $(\phi,0)$ of equation (\ref{eq:mainequation}). Let $u\in C(\R,\dot{H}^1)\cap L^5_tL^{10}_x((-T,T)\times \R^3)$ for any $T\in(0,\infty)$ be the unique solution to equation (\ref{eq:mainequation}) with $\overrightarrow{u}(0)=(u_0,u_1)$. Then there exists $R>0$ and $\delta>0$ such that
\begin{equation}\label{eq:channelofenergy}
\int_{|x|\ge R+|t|} [|\nabla u|^2+(\partial_tu)^2](t,x)\,dx\ge \delta>0,
\end{equation}
for all $t\ge 0$ or all $t\leq 0$.
\end{theorem}

This theorem tells us that if a radial solution to equation (\ref{eq:mainequation}) is not a steady state, 
then it must emit energy to spatial infinity. For applications below we also need the following quantitative version of Theorem~\ref{th:channelofenergy}.

\begin{theorem}\label{th:quantitativechannel}
Take $V\in\Omega\subset Y$. Suppose that radial finite energy initial data $(u_0,u_1)\not\equiv (\phi,0)$ for any steady state solution of equation (\ref{eq:mainequation}), with $\|(u_0,u_1)\|_{\HL(\R^3)}\leq M<\infty$. Let $$u\in C(\R,\dot{H}^1)\cap L^5_tL^{10}_x((-T,T)\times \R^3)$$ for any $T\in(0,\infty)$ be the unique solution to equation (\ref{eq:mainequation}) with $\overrightarrow{u}(0)=(u_0,u_1)$. 
Denote by $\Sigma$  the set of radial steady states of equation (\ref{eq:mainequation}) and define
\begin{equation}\label{eq:distance}
\delta:=\inf\,\left\{\|(u_0,u_1)-(\phi,0)\|_{\dot{H}^1\times L^2}:\,(\phi,0)\in \Sigma\right\}>0.
\end{equation}
Then there exists $c=c(\delta,M, V)>0$ such that
\begin{equation}\label{eq:channelofenergy*}
\int_{|x|\ge |t|} [|\nabla u|^2+(\partial_tu)^2](t,x)\,dx\ge c,
\end{equation}
for all $t\ge 0$ or all $t\leq 0$.
\end{theorem}

\begin{proof} Suppose the theorem fails. Then there exists a sequence of solutions $u_n$ to equation (\ref{eq:mainequation}) with initial data $(u_{0n},u_{1n})\in\dot{H}^1\times L^2(\R^3)$ satisfying $\|(u_{0n},u_{1n})\|_{\HL(\R^3)}\leq M$, (\ref{eq:distance}) with a uniform $\delta>0$, and
\begin{equation}\label{eq:forcontradiction}
\inf_{t\ge 0}\int_{|x|\ge |t|} |\nabla u_n|^2+(\partial_tu_n)^2(t,x)\,dx+\inf_{t\leq 0}\int_{|x|\ge |t|} |\nabla u_n|^2+(\partial_tu_n)^2(t,x)\,dx\leq \frac{1}{n}.
\end{equation}
By passing to a subsequence we can assume that $(u_{0n},u_{1n})$ admits the following profile decomposition
\begin{equation}
(u_{0n},u_{1n})=\overrightarrow{U}^L_1(0)+\sum_{j=2}^J\left(\frac{1}{\lambda_{jn}^{1/2}}U^L_j(-\frac{t_{jn}}{\lambda_{jn}},\frac{x}{\lambda_{jn}}),\frac{1}{\lambda_{jn}^{3/2}}\partial_tU^L_j(-\frac{t_{jn}}{\lambda_{jn}},\frac{x}{\lambda_{jn}})\right)+\overrightarrow{w}_{Jn}(0),
\end{equation}
with the following properties:
\begin{eqnarray*}
&&U^L_j,\,w_{Jn} \,{\rm \,are\,\,radial\,\,and\, \,solve\,\,the\,\, free\, \,wave\,\, equation}\,\,{\rm for\,\,each}\,\,j,\,J;\\
&&{\rm either}\,\,t_{jn}\in \R,\,\lambda_{jn}>0,\,\lim_{n\to\infty}\frac{t_{jn}}{\lambda_{jn}}\in\{\pm\infty\}{\rm \,\,or\,\,}t_{jn}\equiv 0,\,\lim_{n\to\infty}\left(\lambda_{jn}+\frac{1}{\lambda_{jn}}\right)=\infty;\\
&& {\rm for\,}\,j\neq j',\,\lim_{n\to\infty}\left(\frac{\lambda_{jn}}{\lambda_{j'n}}+\frac{\lambda_{j'n}}{\lambda_{jn}}+\frac{|t_{jn}-t_{j'n}|}{\lambda_{jn}}\right)=\infty;\\
&&{write}\,\,w_{Jn}(t,x)=\frac{1}{\lambda_{jn}^{1/2}}\tilde{w}^j_{Jn}(\frac{t-t_{jn}}{\lambda_{jn}},\frac{x}{\lambda_{jn}}),\,\,{then}\,\,\tilde{w}_{Jn}^j\rightharpoonup 0,\,\,{\rm and\,}\,w_{Jn}\rightharpoonup 0,\,\,{\rm as}\,\,n\to\infty;\\
&&\lim_{J\to\infty}\limsup_{n\to\infty}\|w_{Jn}\|_{L^5_tL^{10}_x(\R\times \R^3)}=0.
\end{eqnarray*}
Note that 
\begin{equation*}
\|V\|_{L^{5/4}_tL^{5/2}_x(\{(x,t):\,|x|\ge |t|\})}<\infty,
\end{equation*}
thus by finite speed of propagation we can apply Lemma \ref{lm:profilewithpotential} in the exterior light cone $\{(x,t):\,|x|\ge |t|\}$, and obtain
\begin{equation}\label{eq:decomposition}
\overrightarrow{u}_n(t)=\overrightarrow{U_1}(t)+\sum_{j=2}^J \overrightarrow{U}_{jn}(t)+\overrightarrow{w}_{Jn}(t)+\overrightarrow{r}_{Jn}(t), \quad{\rm for}\,\,|x|\ge |t|
\end{equation}
with 
\begin{equation}\label{eq:errorterm}
\lim_{J\to\infty}\limsup_{n\to\infty}\left(\sup_{t\in \R}\|\overrightarrow{r}_{Jn}(t)\|_{\dot{H}^1\times L^2}+\|r_{Jn}\|_{L^5_tL^{10}_x(\R\times \R^3)}\right)=0,
\end{equation}
where $\overrightarrow{U}_{jn}(t,x)=\left(\frac{1}{\lambda_{jn}^{1/2}}U_j(\frac{t-t_{jn}}{\lambda_{jn}},\frac{x}{\lambda_{jn}}),\frac{1}{\lambda_{jn}^{3/2}}\partial_tU_j(\frac{t-t_{jn}}{\lambda_{jn}},\frac{x}{\lambda_{jn}})\right)$, and with $U_j$ and $U_1$ given as in Lemma \ref{lm:aux} and Lemma \ref{lm:profilewithpotential}. Moreover, denoting $U_{1n}=U_1$, for $\rho_n>\sigma_n>0$ and $\theta_n\in \R$ we have the following orthogonality property for $1\leq j\not= j'$
\begin{eqnarray}
&&\label{eq:no1*}\lim_{n\to\infty}\int_{\sigma_n<|x|<\rho_n}\nabla U_{jn}\nabla U_{j'n}+\partial_tU_{jn}\partial_tU_{j'n}(\theta_n,x)\,dx=0;\label{eq:orth1}\\
&&\label{eq:no2*}\lim_{n\to\infty}\int_{\sigma_n<|x|<\rho_n}\nabla U_{jn}\nabla w_{Jn}+\partial_tU_{jn}\partial_tw_{Jn}(\theta_n,x)\,dx=0\label{eq:orth2}.
\end{eqnarray}
If for some $2\leq j\leq J$, $U_j^L\not\equiv 0$, by results in \cite{JiaLiuXu} (see remark at the end of proof of Lemma~4.6 in \cite{JiaLiuXu}) there exists some fixed $\epsilon>0$ such that 
\begin{equation}\label{eq:nonlinearchannel}
\int_{|x|\ge |t|}|\nabla U_{jn}|^2+(\partial_t U_{jn})^2(t,x)dx\ge \epsilon>0,
\end{equation}
for all $t\ge 0$ or all $t\leq 0$.
Thus by the orthogonality property of profiles (\ref{eq:orth1}),(\ref{eq:orth2}) we have
\begin{eqnarray*}
&&\inf_{t\ge 0}\int_{|x|\ge |t|}|\nabla u_n|^2+(\partial_t u_n)^2(t,x)dx+\inf_{t\leq 0}\int_{|x|\ge |t|}|\nabla u_n|^2+(\partial_t u_n)^2(t,x)dx\\
&&\ge\frac{1}{2}\inf_{t\ge0}\int_{|x|\ge |t|}|\nabla U_{jn}|^2+(\partial_t U_{jn})^2(t,x)dx+\frac{1}{2}\inf_{t\leq 0}\int_{|x|\ge |t|}|\nabla U_{jn}|^2+(\partial_t U_{jn})^2(t,x)dx\\
&&\ge \frac{\epsilon}{2}>0 
\end{eqnarray*}
for all $n$ sufficiently large. A contradiction with (\ref{eq:forcontradiction}). Thus we must have $U_j^L\equiv 0$ for $2\leq j\leq J$. The profile decompositions of $(u_{n0},u_{n1})$ then simplify to this form: 
\begin{equation}
(u_{0n},u_{1n})=\overrightarrow{U}^L_1(0)+\overrightarrow{w}_{n}(0).
\end{equation}
By the decomposition (\ref{eq:decomposition}), Theorem \ref{th:channelofenergy} and orthogonality property of profiles, using  the same arguments as above, we conclude that $\overrightarrow{U}^L_1(0)$ must be a steady state. Thus from the bound (\ref{eq:distance}) we have
\begin{equation}
\|\overrightarrow{w}_{n}(0)\|_{\dot{H}^1\times L^2(\R^3)}\ge\delta.
\end{equation}
By the channel of energy estimates for the linear wave equation, by the decomposition (\ref{eq:decomposition}), and the orthogonality property, we obtain
\begin{eqnarray*}
&&\inf_{t\ge 0}\int_{|x|\ge |t|}|\nabla u_n|^2+(\partial_t u_n)^2(t,x)dx+\inf_{t\leq 0}\int_{|x|\ge |t|}|\nabla u_n|^2+(\partial_t u_n)^2(t,x)dx\\
&&\ge\frac{1}{2}\inf_{t\ge0}\int_{|x|\ge |t|}|\nabla w_{n}|^2+(\partial_t w_{n})^2(t,x)dx+\frac{1}{2}\inf_{t\leq 0}\int_{|x|\ge |t|}|\nabla w_{n}|^2+(\partial_t w_{n})^2(t,x)dx\\
&&\ge \frac{1}{2}\delta>0, 
\end{eqnarray*}
 for all sufficiently large $n$. We thus again arrive at a contradiction with (\ref{eq:forcontradiction}). The theorem is proved.
\end{proof}

\section{Center Stable manifold of  unstable excited states}\label{sec:5}

In this section we show that any unstable excited state can only attract a finite co-dimensional manifold of solutions and finish the proof of our main Theorem \ref{th:maintheoremintro}. We first prove   Theorem~\ref{th:nongeneric} which provides the key estimate on the energy of the radiation term.

Before going into the technical details let us briefly outline the main ideas underlying the proof. By results in Section~\ref{sec:3}, 
 we know that if a solution $U(t)$ scatters to an unstable excited state $(\phi,0)$, then there exists a local finite co-dimensional center-stable  manifold around $\overrightarrow{U}(0)$, on which solutions scatter to the same excited state. We would like to show the following in a small neighborhood of $\overrightarrow{U}(0)$:  if the data $(u_0,u_1)$ does not lie on this local center-stable  manifold, then the solution $\OR{u}(t)$ with initial data $(u_0,u_1)$ will scatter to a steady state with {\em strictly less energy}, thus not to~$(\phi,0)$. Then it is  clear that the set of initial data in $\HL(\R^3)$ for which solutions scatter to an unstable excited state is a global finite co-dimensional manifold. 

Note that the local center-stable  manifold theorem \ref{th:localmanifold} guarantees that the solution $u(t)$ will exit a small ball centered at $\overrightarrow{U}(0)$. However after it exits the small ball, we will lose control on the solution based on perturbative analysis alone. Thus we need some global information about the future development of the solution $u(t)$. The key global information here is the channel of energy inequality. Roughly speaking we show by a channel of energy argument that $u(t)$ will emit at least some fixed amount of energy more than solution starting from the center-stable  manifold to spatial infinity, thus leaving $u(t)$ with less energy in a bounded region as $t\to\infty$ than is required for it  to scatter to $(\phi,0)$. This forces $u(t)$ to scatter to a different steady state from $(\phi,0)$. 
The precise result is as follows.

\begin{theorem}\label{th:nongeneric}
Let $V\in \Omega\subseteq Y$. Suppose that radial finite energy solution $\overrightarrow{U}(t)$ to equation (\ref{eq:mainequation}) scatters to an unstable excited state $(\phi,0)$. Let $\mathcal{M}$ be the local center-stable  manifold around $\overrightarrow{U}(0)$ and let $\epsilon_0,\,\epsilon_1$ be as defined in Theorem \ref{th:localmanifold}. Then there exist $\epsilon$ with $0<\epsilon<\epsilon_1<\epsilon_0$ and $\delta(\epsilon_1)\gg \epsilon $, such that for any solution $u$ with radial finite energy initial data $(u_0,u_1)\notin \mathcal{M}$ with $$\|(u_0,u_1)-\overrightarrow{U_0}\|_{\dot{H}^1\times L^2}<\epsilon,$$ we can find $L>0$ such that 
\begin{equation}\label{eq:emitmoreenergy}
\int_{|x|\ge t-L} \left[\frac{|\nabla u|^2}{2}+\frac{(\partial_tu)^2}{2}\right](t,x)\,dx\ge \mathcal{E}(\overrightarrow{U}(t))-\mathcal{E}((\phi,0))+\delta, \,\,{\rm for}\,\,t\ge L.
\end{equation}
Suppose $\OR{u}(t)$ scatters to $(\phi_{1},0)\in \Sigma$ (the set of steady states). Then 
\begin{equation}
\mathcal{E}((\phi_1,0))<\mathcal{E}((\phi,0)).
\end{equation}
\end{theorem}

\begin{proof}By the local center-stable  manifold theorem of Section~\ref{sec:3}, the locally defined finite co-dimensional manifold $\mathcal{M}$ satisfies the property that any solution to equation (\ref{eq:mainequation}) with initial data on $\mathcal{M}$ scatters to $(\phi,0)$. Moreover, if a solution $\overrightarrow{u}(t)$ with initial data $(u_0,u_1)\in B_{\epsilon_1}(\overrightarrow{U}_0)$ satisfies 
\begin{equation}
\|\overrightarrow{u}(t)-\overrightarrow{U}(t)\|_{\dot{H}^1\times L^2}<\epsilon_1 \,\,{\rm for\,\,all\,\,}t\ge 0,
\end{equation}
then $(u_0,u_1)\in \mathcal{M}$. By shrinking $\epsilon_1$ if necessary we can assume that the distance from any other steady state to $(\phi,0)$ is greater than $3\epsilon_1$. Take $\epsilon<\epsilon_1$ sufficiently small to be chosen below. Since solution $\overrightarrow{U}(t)$ scatters to $(\phi,0)$ as $t\to\infty$, denoting  by $\overrightarrow{U}^L$  the scattered linear wave, we have the property that
\begin{equation}\label{eq:convergencewithradiation}
\lim_{t\to\infty}\|\overrightarrow{U}(t)-\overrightarrow{U}^L(t)-(\phi,0)\|_{\dot{H}^1\times L^2}=0.
\end{equation}
By (\ref{eq:convergencewithradiation}), the fact that $\phi\in\dot{H}^1(\R^3)$ and $U^L\in L^5_tL^{10}_x( [0,\infty)\times \R^3)$, for any small $\delta_1>0$, we can choose $L>0$ and $T_1>L$ sufficiently large such that for $t\ge T_1$ we have: 
\begin{itemize}
\item (Closeness of $\OR{U}$ to $\OR{U}^L+(\phi,0)$ and choice of the bounded region)
\begin{equation}\label{eq:decomp}
\|\overrightarrow{U}(t)-\overrightarrow{U}^L(t)-(\phi,0)\|_{\dot{H}^1\times L^2}+\|\phi\|_{\dot{H}^1(|x|\ge L)}\leq \delta_1;
\end{equation}
\item (Most energy of the free radiation is exterior)
\begin{equation}\label{eq:mostlyexterior}
\int_{|x|\ge t-T_1+L}|\nabla_{x,t}U^L|^2(t,x)\,dx\ge \int_{\R^3}|\nabla_{t,x}U^L|^2(t,x)\,dx-\delta_1;
\end{equation}
\item (Control on the Strichartz norm of the radiation)
Let $$D:=\{(x,t):\,|x|\leq T_1+L-t,\,0\leq t\leq T_1\}.$$ Then we have
\begin{equation}\label{eq:smallfreewave}
\|U^L\|_{L^5_tL^{10}_x(  (0,\infty)\times \R^3\backslash D)}<\delta_1.
\end{equation}
\end{itemize}
We remark that (\ref{eq:mostlyexterior}) ensures that $U^L$ can essentially be taken as zero for our purposes inside the region $|x|\leq t-T_1+L$ for $t\geq T_1$, which will be important to keep in mind
later, in order to  distinguish the second piece of radiation. 

\smallskip 

By the continuous dependence of the solution to equation (\ref{eq:mainequation}) with respect to the initial data in $\HL(\R^3)$ and by finite speed of propagation,  if we take $\epsilon$ sufficiently small and radial initial data $(u_0,u_1)\in\HL\backslash\mathcal{M}$ with $\|(u_0,u_1)-\OR{U}(0)\|_{\HL}<\epsilon$,  then
\begin{equation}
\|\OR{u}(T_1)-\OR{U}(T_1)\|_{\HL}
\end{equation}
can be made sufficiently small. Hence, noting that $\|V\|_{L^{5/4}_tL^{5/2}_x(|x|\ge |t|)}$ is finite, we can apply Lemma \ref{lm:perturbationresult} to conclude that 
\begin{equation}\label{eq:veryclose}
\|\overrightarrow{u}(t)-\OR{U}(t)\|_{\HL(|x|\ge t-T_1)}\leq \delta_1, \,\,{\rm for\,\,all\,\,}t\ge T_1.
\end{equation}
(\ref{eq:veryclose}) means that we can effectively identify $\OR{u}$ with $\OR{U}$ in the exterior region $$|x|\ge t-T_1,\,t\ge T_1.$$ Hence by (\ref{eq:decomp}), we see
\begin{equation}\label{eq:appbound}
\|\OR{u}(t)-\OR{U}^L(t)\|_{\HL(|x|\ge t-T_1+L)}\leq 3\delta_1,
\end{equation}
that is, we can also identify $\OR{u}$ with $\OR{U}^L$ in the exterior region $|x|\ge t-T_1+L,\,t\ge T_1$. 

\smallskip

At this point to avoid any possibility of confusion due to the many parameters, we remark that $\delta_1$ and $\epsilon$ can be made as small as we wish, with $T_1,\,L$ depending on $\delta_1$ and $\OR{U}$ only. $\epsilon$ is a small free parameter below some threshold determined by $\delta_1$. The key point for us is that $\epsilon_1>0$ is fixed no matter how small $\epsilon$ is chosen. 

\smallskip

Since $(u_0,u_1)\not\in \mathcal{M}$, there exists $T_2>0$ such that 
\begin{equation}\label{eq:expelled}
\|\OR{u}(T_2)-\OR{U}(T_2)\|_{\HL(\R^3)}= \epsilon_1.
\end{equation}
Note that the choice of $T_1$ and $L$ does not depend on $\epsilon$, thus by the continuous dependence of solution on initial data in $\dot{H}^1\times L^2$, if we choose $\epsilon$ sufficiently small, we can assume $T_2>L+T_1+1$.
 Let us consider the data $\OR{u}(T_2)$ in more detail.
By estimates (\ref{eq:decomp}) and (\ref{eq:expelled}) we can write
\begin{equation}
\OR{u}(T_2)=(\phi,0)+\OR{U}^L(T_2)+\OR{w},
\end{equation}
where $\OR{w}\in\HL$ satisfies
\begin{equation}
2\epsilon_1\ge\epsilon_1+\delta_1\ge\|\OR{w}\|_{\HL(\R^3)}\ge \epsilon_1-\delta_1\ge\epsilon_1/2,
\end{equation}
if $\delta_1$ is chosen smaller than $\frac{\epsilon_1}{2}$. Consider the solution $\tilde{u}(t)$ to equation (\ref{eq:mainequation}) with $$\OR{\tilde{u}}(T_2)=(\phi,0)+\OR{w}.$$ Then by the quantitative channel of energy inequality from Theorem \ref{th:quantitativechannel}, we infer that 
\begin{equation}\label{eq:anotherchannel}
\int_{|x|\ge |t-T_2|}|\nabla_{t,x} \tilde{u}|^2(t,x)dx\ge c(\epsilon_1)>0, \,\,\,{\rm for\,\,all\,\,}t\ge T_2\,\,\,{\rm or\,\,all\,\,}t\leq T_2.
\end{equation} 
By bound (\ref{eq:smallfreewave}) and Lemma \ref{lm:perturbationresult} we obtain that  for $|x|\ge |t-T_2|$
\begin{equation}\label{eq:anotherdecomp}
\OR{u}(t,x)=\OR{U}^L(t,x)+\OR{\tilde{u}}(t,x)+\OR{r}(t,x),
\end{equation}
where the remainder term $\OR{r}$ satisfies
\begin{equation}\label{eq:errorsmall}
\sup_{t\in \R}\|\OR{r}(t)\|_{\HL(\R^3)}\leq C\delta_1.
\end{equation}
We claim that the channel of energy inequality (\ref{eq:anotherchannel}) holds for all $t\ge T_2$. Otherwise, inequality (\ref{eq:anotherchannel}) holds for all $t\leq T_2$. By (\ref{eq:anotherdecomp}) and (\ref{eq:errorsmall}), setting $t=0$ and noting that $\|U^L\|_{\hl(|x|>T_2)}$ can be made smaller than $\delta$ if $T_2$ is large enough, we see that 
\begin{equation}
\|(u_0,u_1)\|_{\HL(|x|\ge T_2)}\ge c(\epsilon_1)-C\delta_1>\frac12c(\epsilon_1) ,
\end{equation}
if $\delta_1$ is chosen sufficiently small, a contradiction with finiteness of $\overrightarrow{U}(0)$ in $\HL$  and $\|(u_0,u_1) - \overrightarrow{U}(0)\|_{\hl}<\epsilon$ for $T_2$ large. Thus we have the following channel of energy inequality
\begin{equation}
\int_{|x|\ge t-T_2}|\nabla_{t,x} \tilde{u}|^2(t,x)\,dx\ge c(\epsilon_1)>0, \,\,{\rm for\,\,all\,\,}t\ge T_2.
\end{equation}
The estimate (\ref{eq:appbound}), decomposition (\ref{eq:anotherdecomp}) and estimate on the remainder term (\ref{eq:errorsmall}) imply for $t\ge T_2$
\begin{equation}
\int_{|x|\ge t-T_1+L}|\nabla_{t,x}\tilde{u}|^2(t,x)\, dx\leq C\delta_1,
\end{equation}
and consequently
\begin{equation}
\int_{t-T_1+L\ge |x|\ge t-T_2}|\nabla_{t,x}\tilde{u}|^2(t,x)\, dx\ge c(\epsilon_1)-C\delta_1.
\end{equation}
Hence
\begin{eqnarray*}
&&\|\OR{u}\|^2_{\HL(|x|\ge t-T_2)}\\
&&\ge \|\OR{U}^L\|^2_{\HL(|x|\ge t-T_1+L)}+c(\epsilon_1)-C\delta_1\\
&&\ge \|\OR{U}^L\|^2_{\HL(\R^3)}+c(\epsilon_1)-C\delta_1,
\end{eqnarray*}
if $\delta_1$ is chosen sufficiently small. Choose $\delta_1$ sufficiently small depending on $\epsilon_1$, (\ref{eq:emitmoreenergy}) is then proved with some $\delta=\delta(\epsilon_1)>0$.

\smallskip

\noindent

Now let $\OR{u}^L(t)$ be the free radiation term for the solution $\OR{u}(t)$ as $t\to\infty$. Then
\begin{equation}
\lim_{t\to\infty}\|\OR{u}(t)-\OR{u}^L(t)-(\phi_{1},0)\|_{\HL(\R^3 )}=0.
\end{equation}
By the first part of Theorem \ref{th:nongeneric},
\begin{equation}\label{eq:moreradiatedenergy}
\|\OR{u}^L\|^2_{\HL(\R^3)}\ge \lim_{t\rightarrow \infty }\|\OR{u}\|^2_{\HL(|x|\ge t-T_2)}\ge \|\OR{U}^L\|^2_{\HL(\R^3)}+c(\epsilon_1).
\end{equation}
Note that
\begin{eqnarray*}
&&\|\OR{U}^L\|^2_{\HL}=\mathcal{E}(\OR{U})-\mathcal{E}(\phi,0);\\
&&\|\OR{u}^L\|^2_{\HL}=\mathcal{E}(\OR{u})-\mathcal{E}(\phi_1,0).
\end{eqnarray*}
If $\|\OR{u}(0)-\OR{U}(0)\|_{\HL(\R^3)}$ is chosen small, by (\ref{eq:moreradiatedenergy}) we conclude
\begin{equation}
\mathcal{E}(\phi_1,0)<\mathcal{E}(\phi,0).
\end{equation}
The theorem is proved.\end{proof}

\begin{proof}[Proof of Theorem \ref{th:maintheoremintro}.]
We only consider the case in which $(\phi,0)$ is unstable, the case of $(\phi,0)$ being stable can be handled using standard perturbation arguments. 
Since in some small neighborhood of any point $\OR{U}(0)$ on $\mathcal{M}_{\phi}$, $\mathcal{M}_{\phi}$ coincides with the local center-stable  manifold $\mathcal{M}$ of codimension $n$ which we constructed in Section~\ref{sec:3} by Theorem \ref{th:nongeneric}, $\mathcal{M}_{\phi}$ is thus a global manifold of co-dimension~$n$. The path-connectedness follows from the following theorem. 
\end{proof}

 \begin{theorem} For any unstable excited state $(\phi, 0)$, the corresponding center-stable  manifold $\mathcal{M}_{\phi}$ is path connected. 
 \end{theorem}

\begin{proof}
Given data $(u_0, u_1), (\tilde{u}_0,\tilde{u}_1)\in \mathcal{M}_{\vp}$, we denote the corresponding  solutions by $u, \tilde{u}$. Write $h=u-\vp, \ell=\tilde{u}-\vp$. 
Repeat step~0 and step~1  in  the proof of Theorem~\ref{th:localmanifold}. Then given any $\epsilon \ll 1$, we can find $T=T(\epsilon, u, \tilde{u})$, such that 
\begin{equation}\|h\|_{L^2_tL^\infty_x\cap L^\infty_tL^6_x(  [T,\infty)\times \R^3)}, \hspace{0.1cm}\|  \ell \|_{L^2_tL^\infty_x\cap L^\infty_tL^6_x(  [T,\infty)\times \R^3)}\leq \epsilon \label{h-l-small}.\end{equation}  
Now we seek a function $w(\theta, t, x)$ of the form 
\begin{align}w(\theta, t, x) &=(1-\theta) u +\theta\tilde{u}+\eta \notag \\ &= \vp + (1-\theta) h +\theta \ell +  \sum_{i=1}^n  \lambda_i(\theta, t)\rho_i +\gamma(\theta, t,x) \label{w-formula}\end{align}
such that  for all $\theta\in [0,1]$, $\gamma(\theta, t,x)\perp \rho_i, i=1,\ldots, n$ and  $w(\theta,t,x)$ is a solution to equation (\ref{eq:mainequation}) that scatters to $\vp$. 

For $\theta\in[0,1]$ fixed,   the equation satisfied by $\eta = \sum_{i=1}^n  \lambda_i(\theta, t)\rho_i +\gamma(\theta, t,x)$ is: 
\[\eta_{tt}-\Delta\eta -V(x)\eta + 5\vp^4\eta + N(\theta,h,\ell,\vp,\eta)=0,\]
where 
\[N(\theta,h,\ell,\vp,\eta) = (\vp + (1-\theta) h +\theta \ell +\eta )^5 -(1-\theta) (\vp+h)^5 -\theta (\vp +\ell)^5 - 5\vp^4 \eta.\]
Now  we  can repeat the stability condition (\ref{condition2}) and obtain  the reduced system of the form (\ref{system-2}).

In $N(\theta,h,\ell,\vp,\eta)$, the terms independent of $\eta$ are of the form 
\begin{align*}(\vp + (1-\theta) h +\theta \ell   )^5 -(1-\theta) (\vp+h)^5 -\theta (\vp +\ell)^5
=\sum_{i+j+k=5, i\leq 3} C(\theta, i, j, k) \vp^i h^j \ell^k. 
 \end{align*}
Notice that there are no  terms  $\vp^5$ or $\vp^4 h,  \phi^4\ell$. 
 
 Also, the linear term of $\eta$ in  $N(\theta,h,\ell,\vp,\eta)$ is
 \[5(\vp + (1-\theta) h +\theta \ell )^4 \eta   - 5\vp^4 \eta\]
hence  all linear terms involve a factor of $h$ or $\ell$.

Now we can repeat estimates (\ref{tilde-lambda})(\ref{tilde-gamma}),  then (\ref{N1-estimate}) for the linear term in $\eta$, (\ref{N2-estimate}) for higher order terms in $\eta$.  
We also have  the following estimate on terms independent of $\eta$
\[\left\|\sum_{i+j+k=5, i\leq 3} C(\theta, i, j, k)  \vp^i h^j \ell^k\right\|_{\st{1}{2}( [T,\infty)\times \R^3)}\lesssim \epsilon^2  \] 
For example, one checks that  $$ \|  \vp^3 h^2\|_{\st{1}{2}( [T,\infty)\times \R^3)}\leq \|\phi\|_{L^6}^3\|h\|^2_{L^{2}_tL_x^{\infty}}\leq 
 \epsilon^2$$ using (\ref{h-l-small}).  To sum up, using the $X$ norm defined in (\ref{define-X}), we conclude that 
 \begin{align*}\|(\lambda_1,\cdots, \lambda_n, \gamma)\|_{X([T,\infty))}\leq & \, K\epsilon^2 +  K\left( \sum_{i=1}^n |\lambda_i(\theta, T)| + \|(\gamma(\theta, T), \dot{\gamma}(\theta,T))\|_{\hl} \right)
\\ &+ K \epsilon \|(\lambda_1,\cdots, \lambda_n, \gamma)\|_{X([T,\infty))} + K \sum_{k=2}^5\|(\lambda_1,\cdots, \lambda_n, \gamma)\|_{X([T,\infty))}^k,\end{align*}
where $K$ is some absolute constant.  

Moreover, in a similar fashion one sees that the difference of two solutions satisfies a similar estimate in which   the first two terms disappear.  Using the contraction mapping principle, we   conclude that for  sufficiently small  data
\[ \sum_{i=1}^n |\lambda_i(\theta,T)| + \|(\gamma(\theta,T), \dot{\gamma}(\theta, T))\|_{\hl}\leq \delta\]
there is  a solution  $w$ as in (\ref{w-formula})   which solves (\ref{eq:mainequation}). We can also  check that $w$ scatters to $\phi$.
 
In particular, let us take 
$\lambda_i(\theta,T)=\frac{1}{n}\delta \theta(1-\theta)$ and $\vec{\gamma}(\theta,T,x)=\vec{0}$. We claim that the corresponding solution $w(\theta, t, x)$  satisfies the following relation
\begin{equation} w(0,t,x)=u(t,x),\quad w(1,t,x)=\tilde{u}(t,x),\quad \text{ for all } t\in \R. \label{w-equal}\end{equation}
In fact, notice that $\lambda_i(0,T)=0, \vec{\gamma}(0,T,x)=\vec{0}$ implies $\lambda_i(0,t)=0, \vec{\gamma}(0,t,x)=\vec{0}$ for $t\geq T$, which further implies $w(0,t,x)=u(t,x), t\geq T$. Similarly we have $w(1,t,x)=\tilde{u}(t,x), t\geq T$. Then (\ref{w-equal}) follows from the  uniqueness of solutions to equation (\ref{eq:mainequation}).

Hence $\{\vec{w}(\theta, 0, x), \theta\in [0,1]\}$ is a path in $\mathcal{M}_{\phi}$ connecting the two data $(u_0, u_1), (\tilde{u}_0, \tilde{u}_1)$.
\end{proof}

\begin{section}{Appendix A: Some elliptic estimates.}

We begin with the following lemma.

\begin{lemma}\label{lm:ellipticlemma}
Denote by $B_1$  the ball of radius $1$ in $\R^3$. Let $V\in L^{\infty}(B_1)$ and $\lambda>0$. Suppose $u\in \dot{H}^1\cap L^6(B_1)$ is a weak solution to
\begin{equation}\label{eq:ellipticapp}
-\Delta u-Vu+\lambda u^5=0,\quad {\rm in}\,\,B_1,
\end{equation}
in the sense of distributions. Then $u\in L^{\infty}(B_{\frac{1}{2}})$ and
\begin{equation}
\|u\|_{L^{\infty}(B_{\frac{1}{2}})}\leq C(\|V\|_{L^{\infty}(B_1)})\|u\|_{L^2(B_1)}.
\end{equation}
\end{lemma}

\begin{remark} The assumption on the regularity of $V$ can be significantly relaxed to $V\in L^q$ with $q>\frac{3}{2}$. For the sake of simplicity, we shall not prove the most general version. The proof is based on DeGiorgi-Nash iteration arguments.
\end{remark}

\begin{proof}Assume $u$ is not identically zero. By multiplying a positive constant to $u$, we can assume in addition that $\|u\|_{L^2(B_1)}=1$. $u$ still satisfies equation (\ref{eq:ellipticapp}) with possibly a different $\lambda$ in the equation. Fixing an $M>1$ to be determined later, we need to show $\|u\|_{L^{\infty}(B_{\frac{1}{2}})}\leq M$. For each integer $k\ge 1$, we define 
\begin{eqnarray*}
&&r_k=\frac{1}{2}+\frac{1}{2^k};\\
&&c_k=(1-\frac{1}{2^k})M.
\end{eqnarray*}
For each $k\ge 2$, we fix a smooth nonnegative cutoff function $\eta_k$ such that $\eta_k|_{B_{r_k}}\equiv 1$ and ${\rm supp}\,\eta_k\Subset B_{\frac{r_k+r_{k-1}}{2}}$, with $0\leq \eta_k\leq 1$ and $|\nabla \eta_k|\leq C2^k$. By a sign change, it suffices to prove that $u(x)\leq M$ for $x\in B_{\frac{1}{2}}$. Multiplying equation (\ref{eq:ellipticapp}) with $(u-c_k)_+\eta_k$ and integrating, we get that
\begin{eqnarray*}
&&\int_{B_1}\nabla (u-c_k)_+\nabla\left((u-c_k)_+\eta_k\right)\,dx-\int_{B_1}V(u-c_k)_+\cdot(u-c_k)_+\eta_k\,dx\\
&&\quad\quad-\int_{B_1}c_kV(u-c_k)_+\eta_k\,dx+\lambda\int_{B_1}u^5 (u-c_k)_+\eta_k\,dx=0,
\end{eqnarray*}
which implies that
\begin{eqnarray*}
&&\int_{B_{r_k}}|\nabla (u-c_k)_+|^2+(u-c_k)_+^2\,dx\\
&&\leq C(\|V\|_{L^{\infty}(B_1)}+1)\int_{B_{\frac{r_k+r_{k-1}}{2}}}(u-c_k)_+^2\,dx+\\
&&\quad\quad+\|V\|_{L^{\infty}(B_1)}c_k\int_{B_{\frac{r_k+r_{k-1}}{2}}}(u-c_k)_+\,dx+C4^k\int_{B_{\frac{r_k+r_{k-1}}{2}}}(u-c_k)_+^2\,dx.
\end{eqnarray*}
Thus,
\begin{eqnarray*}
&&\left(\int_{B_{r_k}}(u-c_k)_+^6\,dx\right)^{\frac{1}{3}}\\
&&\leq C(\|V\|_{L^{\infty}(B_1)}+1)4^k\int_{B_{r_{k-1}}}(u-c_{k-1})_+^2\,dx+\\
&&\quad+\|V\|_{L^{\infty}}c_k\left(\int_{B_{\frac{r_k+r_{k-1}}{2}}}(u-c_k)_+^2\,dx\right)^{\frac{1}{2}}\left|\{x\in B_{\frac{r_k+r_{k-1}}{2}}:\,u(x)>c_k\}\right|^{\frac{1}{2}}.
\end{eqnarray*}
Note that by Chebyshev inequality we have
\begin{eqnarray*}
\left|\{x\in B_{\frac{r_k+r_{k-1}}{2}}:\,u(x)>c_k\}\right|\leq \frac{1}{(c_k-c_{k-1})^2}\int_{B_{r_{k-1}}}(u-c_{k-1})_+^2\,dx.
\end{eqnarray*}
Hence, we get that
\begin{eqnarray*}
&&\left(\int_{B_{r_k}}(u-c_k)_+^6\,dx\right)^{\frac{1}{3}}\\
&&\leq C(\|V\|_{L^{\infty}(B_1)}+1)4^k\int_{B_{r_{k-1}}}(u-c_{k-1})_+^2\,dx+\\
&&\quad+\|V\|_{L^{\infty}}\frac{c_k}{c_k-c_{k-1}}\int_{B_{r_{k-1}}}(u-c_{k-1})_+^2\,dx\\
&&\leq C(\|V\|_{L^{\infty}(B_1)}+1)4^k\int_{B_{r_{k-1}}}(u-c_{k-1})_+^2\,dx.
\end{eqnarray*}
Note that
\begin{equation*}
\int_{B_{r_k}}(u-c_k)_+^6\,dx\ge \int_{B_{r_{k+1}}}(u-c_{k+1})_+^2\cdot(c_{k+1}-c_k)^4\,dx.
\end{equation*}
Summarizing the above inequalities, we obtain
\begin{equation}\label{eq:iterationinequality}
\left(\int_{B_{r_{k+1}}}(u-c_{k+1})_+^2\,dx\right)^{\frac{1}{3}}\leq C(\|V\|_{L^{\infty}(B_1)}+1)2^{4k}\int_{B_{r_{k-1}}}(u-c_{k-1})_+^2\,dx, \quad {\rm for\,\,all\,\,}k\ge 2.
\end{equation}
Denote 
\begin{equation}
\epsilon_k:=\left(\int_{B_{r_{k}}}(u-c_{k})_+^2\,dx\right)^{\frac{1}{3}},
\end{equation}
then equation (\ref{eq:iterationinequality}) can be written as
\begin{equation}\label{eq:iteapp}
\epsilon_{k+1}\leq C_116^{k}\epsilon_{k-1}^3\leq (C_116)^{k}\epsilon_{k-1}^3,
\end{equation}
where we have suppressed the dependence of $C_1>1$ on $V$.
Recall that a routine energy inequality implies $\|\nabla u\|_{L^2(B_{\frac{3}{4}})}\leq C\|u\|_{L^2(B_1)}=C$. Hence, if we choose $M$ sufficiently large, by the  H\"older and Chebyshev inequalities, we can assume $\epsilon_2,\,\epsilon_3$ satisfy 
\begin{equation}
\epsilon_2,\,\epsilon_3\leq C\frac{1}{M^{\frac{1}{3}}}\leq \frac{1}{(48C_1)^{6}}.
\end{equation}
Then from the iterative inequality (\ref{eq:iteapp}) we can prove by induction that $\epsilon_k\leq \frac{1}{(48C_1)^{k+3}}$ for all $k\ge 2$.
Hence we have
\begin{equation}
\lim_{k\to\infty}\int_{B_{r_k}}(u-c_k)_+^2\,dx=0.
\end{equation}
Note that $r_k\to \frac{1}{2}$ and $c_k\to M$ as $k\to\infty$, thus we conclude $u(x)\leq M$ for $x\in B_{\frac{1}{2}}$. The lemma is proved.\end{proof}

Now we are ready to prove the main result on the regularity and decay of steady states to equation (\ref{eq:mainequation}).

\begin{theorem}
Let $V\in Y$. Suppose $u\in \dot{H}^1(\R^3)\cap L^6(\R^3)$ is a steady state solution to equation (\ref{eq:mainequation}), i.e., $u$ solves
\begin{equation}\label{eq:ellipticalequationapp}
-\Delta u-Vu+u^5=0\quad{\rm in}\,\,\R^3,
\end{equation}
in the sense of distributions. Then $u\in W^{2,p}_{{\rm loc}}(\R^3)$ for any $1<p<\infty$, and
\begin{equation}\label{eq:decayapp}
|u(x)|\leq \frac{C}{(1+|x|)}, \quad{x\in \R^3.}
\end{equation}
\end{theorem}

\begin{proof} By Lemma \ref{lm:ellipticlemma}, we see $u\in L^{\infty}(\R^3)$. Then the $W^{2,p}_{{\rm loc}}$ estimates follow from standard elliptic regularity theory. Let us now turn to the proof of the decay estimate (\ref{eq:decayapp}). For any $2R=|x_0|>2$, set $v(x):=R^{\frac{1}{2}}u(Rx+x_0)$. Then $v$ solves 
\begin{equation*}
-\Delta u-V_{R,x_0}u+u^5=0,\quad {\rm in}\,\,B_1,
\end{equation*}
where $V_{R,x_0}=R^2V(Rx+x_0)$. Since $V\in Y$ with $\sup\limits_{x\in \R^3}(1+|x|)^{\beta}|V(x)|<\infty$ for some $\beta>2$, we see that $\|V_{R,x_0}\|_{L^{\infty}(B_1)}\leq \|V\|_Y$. Thus by Lemma \ref{lm:ellipticlemma}, 
\begin{equation*}
|v(0)|\leq \|v\|_{L^2(B_1)}\leq C\|v\|_{L^6(B_1)}\leq C\|u\|_{L^6(\{|x|\ge R/2\})}. 
\end{equation*}
Hence by rescaling and the fact that $x_0$ is arbitrary, we see that
\begin{equation*}
|u(x)|=o\left( \frac{1}{(1+|x|)^{\frac{1}{2}}}\right),\,\,\,{\rm as}\,\,|x|\to\infty.
\end{equation*}
Now choose $R>1$ sufficiently large, and set $v(x)=R^{\frac{1}{2}}u(Rx)$, then $v|_{\partial B_1}=R^{\frac{1}{2}}u|_{\partial B_R}$ is small, and $V_R:=R^2V(Rx)$ is small in the exterior of $B_1$, in the sense that
\begin{equation*}
\sup_{|x|\ge 1}|(1+|x|)^{\beta}V_R(x)|<\epsilon_R\to 0,\quad{\rm as}\,\,R\to\infty.
\end{equation*}
Then for sufficiently large $R$, by a standard perturbation argument one constructs a solution $\tilde{v}$ to the equation
\begin{equation}\label{eq:exteriorapp}
-\Delta \tilde{v}-V_R\tilde{v}+\tilde{v}^5=0,\quad {\rm in}\,\,B_R^c,
\end{equation}
with the boundary condition  $\tilde{v}|_{\partial B_1}=v|_{\partial B_1}$, with the estimates 
\begin{equation}\label{eq:preliminarydecayapp}
|\tilde{v}(x)|\leq \frac{c_R}{(1+|x|)},\quad{\rm for}\,\,|x|\ge 1.
\end{equation}
Since both $v$ and $\tilde{v}$ are small solutions in $\dot{H}^1$ to equation (\ref{eq:exteriorapp}) with the same boundary condition, we conclude $v=\tilde{v}$. Hence $v$ also satisfies (\ref{eq:preliminarydecayapp}). Scaling back to $u$, we see that $u$ satisfies
\begin{equation}
|u(x)|\leq\frac{C}{(1+|x|)}, \quad{\rm for}\,\,x\in \R^3.
\end{equation}
The theorem is proved.\end{proof}
\end{section}

\begin{section}{Appendix B: Endpoint Radial Strichartz Estimate}
In this appendix, we  give a proof of Theorem~\ref{Strichartz} in the radial setting. 
We only need to  prove the endpoint case, since other cases are known from~\cite{GV}. In the case of the homogeneous equation, Klainerman and Machedon~\cite{KM} first observed that the endpoint case $(q, r)=(2,\infty)$ holds true for data of the form $(v(0), v_t(0)) = (0, g)$ for radial function $g$. Here we extend the result to the inhomogeneous equation with general radial data. 
\begin{theorem}
Let $v$ be a finite energy solution to the 3d wave equation 
\[ (\partial_{tt}-\Delta) v =F\]
with initial data $(v(t_0), \partial_t v(t_0))= (f, g)\in \hl(\R^3)$ and assume that $f, g, F$ are radially symmetric. Then we have the estimate
\begin{equation}\label{Strichartz-Appendix}
  \|v\|_{\st{2}{\infty} (\R\times \R^3)} \leq C(q,r) \left(\|(f,g)\|_{\hl}  +\|F\|_{\st{1}{2}(\R \times \R^3)}\right),
\end{equation}
\end{theorem}
\begin{proof}
The solution takes the form
\[v(t,x) = \cos (|\nabla|t)f+\frac{1}{|\nabla|}\sin (|\nabla|t)g +\int_0^t\frac{\sin (|\nabla|(t-s))}{|\nabla|}F (s)ds. \]
Case 1: $f=0, F=0$, this was proved in~\cite{KM} using the explicit formula 
\[v(t,r)=\frac{c}{r}\int_{|t-r|}^{t+r}  g(\rho) \rho d\rho , \hspace{1cm} t>0.\] 
Denote $M(\cdot)$ to be the 1d maximal function defined as $$M(f)(x)=\sup_{x\in I} \frac{1}{|I|}\int_I |f|$$ for any  interval $I\subset \R$. Extending $g$ to be an even function defined on $\R$, we then have 
\begin{align*}\sup_{r>0}|v| &\leq C \sup_{r>0}\left|\frac{1}{r}\int_{|t-r|}^{t+r}  |g(\rho) \rho| d\rho\right| \leq CM[  g(\rho)\rho](t).\end{align*}
This is obvious for   $r<t$, while for $r>t$, notice that $[t-r, t+r]$ is a larger interval than $[r-t, t+r]$, hence the term in the middle is still bounded by maximal function.  

Using the Hardy-Littlewood maximal inequality~\cite{Stein}, we get that 
\[\|\sup_{r>0} |v| \|_{L^2_t(\R^{+})}\leq C \|M[g(\rho)\rho](t)\|_{L^2_t(\R^{+})}\leq C\|g(\rho)\rho\|_{L^2(\R)} = C\|g\|_{L^2_x(\R^3)}.\]
By time reversibility, we have the same estimate when integrating on $t\in (-\infty, 0]$.

We rewrite the estimate in the following form
\begin{equation}\left\|\frac{1}{|\nabla|}\sin (|\nabla|t)g\right\|_{\st{2}{\infty} (\R\times \R^3)} \lesssim \|g\|_{L^2(\R^3)}. \label{st-g}\end{equation}
Case 2: $f=g=0$. Using (\ref{st-g})
 and Minkowski inequality, we immediately get 
\begin{align*}\left\|\int_0^t \frac{\sin (|\nabla|(t-s))}{|\nabla|}F (s)ds\right\|_{\st{2}{\infty} (\R\times \R^3)} & \lesssim 
\int_0^\infty \left\|\chi_{[0,t]}(s)\frac{\sin (|\nabla|(t-s))}{|\nabla|}F (s)\right\|_{\st{2}{\infty} (\R\times \R^3)} ds\\
& \lesssim \int_0^\infty \|F(s)\|_{L^2(\R^3)}ds = \|F\|_{\st{1}{2}(\R \times \R^3)}.
\end{align*}
Case 3: $g=0, F=0$.  For simplicity, we prove our estimates for Schwartz function $f$, and the estimates for general $f\in \dot{H}^1_{rad}$ follow by approximation. 

In the radial case, we have the following explicit expression for $v$:
\begin{align}\label{v-formula}
v= \partial_t \frac{1}{r}\int_{|t-r|}^{r+t}f(\rho)\rho \,dt = \frac{1}{r}[f(r+t)(r+t) - f(|r-t|)(t-r)],\hspace{1cm} t>0.
\end{align}
As a first step we bound
\begin{align*}
\sup_{0<r<t}|v(t,r)|\leq \sup_{0<r<t} \frac{1}{r}\left|\int_{t-r}^{t+r} (\rho f(\rho))' d\rho\right|\lesssim M [(\rho f(\rho))'](t).
\end{align*}
Applying Hardy-Littlewood maximal inequality and Hardy's inequality~\cite{Stein}, we get 
\begin{align}\|\sup_{0<r<t} |v|\|_{L^2_t(\R^{+})}\lesssim&  \|f(\rho)+\rho f'(\rho)\|_{L^2(\R^{+})} \notag \\ 
\lesssim & \left\|\frac{f}{|x|}\right\|_{L^2(\R^3)}+ \|\nabla f\|_{L^2(\R^3)} \lesssim \|\nabla f\|_{L^2(\R^3)}. \label{r-small}
\end{align}
Next if $r>t$ we claim that 
\begin{align}
\|\sup_{\rho>t}|f(\rho)|\|_{L^2_t(\R^{+})}^2 \lesssim \int_0^\infty |f'(\rho)|^2\rho^2 d\rho. \label{star}
\end{align}
Dualizing (\ref{star}) we see that it is equivalent to 
\begin{align}
\left|  \iint_{\rho>t>0}\int_{\rho}^\infty f'(w)\,dw \,  h(t,\rho)\, d\rho\, dt \right| & = \left| \int_0^\infty w f'(w)\left[ \frac{1}{w} \iint_{w>\rho>t>0} h(t, \rho) \,d \rho \, dt \right ]d w\right| \notag \\ &\lesssim \|f'(w)w\|_{L^2_w(\R^{+})}\|h\|_{L^2_tL^1_\rho(\R^{+}\times \R)}.
\end{align}
So we need to  prove  that 
\[
\left\|\frac{1}{w}\iint_{w>\rho>t>0} h(t,\rho)\,d\rho \,dt\right\|_{L^2_{w}(\R^{+})}\lesssim \|h\|_{L^2_tL^1_\rho(\R^{+}\times \R)}, \label{dagger}
\]
by  change of variables  $t=w\tau$, we have
\begin{align*}
\left\|\iint_{w >\rho > w\tau>0} |h(w\tau, \rho)| \,d\rho \,d\tau \right\|_{L^2_w(\R^{+})}\lesssim&  \left\|\int_0^1 \int_0^\infty |h(w\tau, \rho)|\, d\rho \, d\tau\right\|_{L^2_w(\R^+)}\\
\lesssim & \int_0^1\left \|\int_0^\infty |h(t,\rho)| d\rho\right\|_{L^2_t(\R^{+}) }\frac{d \tau}{\sqrt{\tau}} \\ \lesssim &\|h\|_{L^2_tL^1_\rho(\R^{+}\times \R)}
\end{align*}
as we claimed. Thus we proved (\ref{star}) whence 
\begin{equation}\left\|\sup_{r>t>0}\frac{1}{r}|f(r+t)|(r+t)\right\|_{L^2_t(\R^+)}\lesssim \|\nabla f\|_{L^2_x}.\label{r-big1}\end{equation}
Next we will prove that 
\begin{equation}\left\|\sup_{r>t>0}\frac{1}{r}|f(r-t)|(r-t)\right\|_{L^2_t(\R^{+})}\lesssim \|\nabla f\|_{L^2_x}.\label{r-big2}\end{equation}
By the change of variables $r=t(\rho+1)$, this is equivalent to 
\[
\left\|\sup_{\rho>0}\frac{\rho}{1+\rho}|f(t\rho)|\right\|_{L^2_t(\R^{+})}\lesssim\left (\int_0^\infty |f'(\rho)|^2\rho^2 d\rho\right)^{\frac12}.
\]
By duality, this is equivalent to 
\[
\left|\int_0^\infty \int_0^\infty \frac{\rho}{1+\rho} \int_{t\rho}^\infty  f'(w)\, dw \, \,h(t,\rho)\,d\rho\, dt\right|\lesssim \|f'(w)w\|_{L^2_w(\R^+)} \|h\|_{L^2_tL^1_\rho(\R^+\times \R)},
\]
or 
\[
\left|\int_0^\infty  w  f'(w)  \left[\frac{1}{w}\iint_{w>t\rho>0}\frac{\rho}{1+\rho} h(t,\rho)\, d\rho \,dt\right] dw\right|\lesssim  \|f'(w)w\|_{L^2_w(\R^+)} \|h\|_{L^2_tL^1_\rho(\R^+\times \R)}.\]
So we only need to show that 
\[\left\|\frac{1}{w}\iint_{w>t\rho>0}\frac{\rho}{1+\rho} h(t,\rho)\, d\rho\, dt\right\|_{L^2_w(\R^+)} \lesssim \|h\|_{L^2_tL^1_\rho(\R^+\times \R)}.\]
By the change of variables $t=\tau w$
\begin{align*}
&\left\|\frac{1}{w}\iint_{w>t\rho}\frac{\rho}{1+\rho} h(t,\rho)\,d\rho\, dt\right\|_{L^2_w(\R^+)}\\
=& \left\| \iint_{1>\tau\rho}\frac{\rho}{1+\rho} h(\tau w,\rho)\,d\rho\, d\tau\right\|_{L^2_w(\R^+)}
\\
\lesssim& \left\|\int_1^\infty \int_0^{\frac{1}{\rho}} |h(\tau w,\rho)| \,d\tau \,d\rho\right\|_{L^2_w(\R^+)} + \left\|\int_0^1\int_0^{\frac{1}{\rho}} \rho |h(\tau w,\rho)|   \,d\tau\, d\rho\right \|_{L^2_w(\R^+)}\\
\lesssim & \left\|\int_0^1 \int_0^{\infty} |h(\tau w,\rho)|  \,d\rho \,d\tau\right\|_{L^2_w(\R^+)} +\left\|\int_0^\infty \int_0^{\min(\frac{1}{\tau}, 1)}\rho |h(\tau w,\rho)|\,d\rho \,d\tau \right\|_{L^2_w(\R^+)}\\
\lesssim &\int_0^1 \|h\|_{L^2_tL^1_\rho(\R^+\times \R)} \frac{d\tau}{\sqrt{\tau}}+\left \|\int_0^\infty \int_0^{\infty}\frac{1}{1+\tau} |h(\tau w,\rho)| \, d\rho \,d\tau\right \|_{L^2_w(\R^+)}\\
\lesssim & \|h\|_{L^2_tL^1_\rho(\R^+\times \R)}+\int_0^\infty \frac{1}{1+\tau}\tau^{-\frac12}d\tau \|h\|_{L^2_tL^1_\rho(\R^+\times \R)}
\\ \lesssim &\|h\|_{L^2_tL^1_\rho(\R^+\times \R)}
\end{align*}
Hence combining (\ref{v-formula}) and (\ref{r-big1}),  (\ref{r-big2}), we deduce
\[\|\sup_{r>t} |v|\|_{L^2_t(\R^+)}\lesssim  \|\nabla f\|_{L^2(\R^3)}. \]
Together with (\ref{r-small}) we have proved
\[\|v\|_{L^2_tL^\infty_x(\R^+\times \R^3)}\lesssim \|\nabla f\|_{L^2(\R^3)}. \]
By   time reversibility, we  get
\[\|v\|_{L^2_tL^\infty_x(\R\times \R^3)}\lesssim \|\nabla f\|_{L^2(\R^3)}\]
 and we are done. 
\end{proof}
 
\end{section}

\bigskip

\noindent
{\sl Acknowledgement:} The proof of the uniqueness of ground states up to a change of sign was communicated to us by Tianling Jin.


\begin{thebibliography}{9}

\bibitem{agmon} S.\ Agmon, \emph{Lectures on exponential decay of solutions of second-order elliptic equations: bounds on eigenfunctions of $N$-body Schr\"odinger operators}, Princeton University Press, Princeton, NJ, 1982.

\bibitem{agmon2} S.\ Agmon, \emph{Spectral properties of Schr\"odinger operators and scattering theory}, 
Ann.\ Scuola Norm.\ Sup.\ Pisa Cl.\ Sci.\ (4) 2 (1975), no.~2, 151--218. 

\bibitem{BaGe} H.\ Bahouri, P.\ G\'erard, \emph{High frequency approximation of solutions to critical nonlinear wave equations}, American Journal of Mathematics, Vol.\ 121, no.~1 (1999), 131--175.

\bibitem{Marius} M.\ Beceanu, \emph{Structure of wave operators for a scaling-critical class of potentials},
Amer.\ J.\ Math.\ 2014 Vol.\ 136, no.~2, 255--308. 



\bibitem{Cotemap2} R.\ Cote, C.\ Kenig, A.\ Lawrie, W.\ Schlag, \emph{Characterization of large energy solutions of the equivariant wave map problem: II}, American Journal of Mathematics. 137 (2015), no.~1, 209--250.
 

\bibitem{Cotemap1} R.\ Cote, C.\ Kenig, A.\ Lawrie, W.\ Schlag, \emph{Characterization of large energy solutions of the equivariant wave map problem: I},  American Journal of Mathematics. 137 (2015), no.~1, 139--207.
 


\bibitem{RKS} R.\ Cote, C.\ Kenig, W.\ Schlag, \emph{Energy partition for the linear wave equation},
Mathematische Annalen. 358 (2014), 573--607.

\bibitem{cotesoliton} R.\ Cote, \emph{Soliton resolution for equivariant wave maps to the sphere}, arXiv:1305.5325, to appear in 
Comm.\ Pure.\ Appl.\ Math.

\bibitem{DKM1}  T.\ Duyckaerts, C.\ Kenig, F.\ Merle, \emph{Universality of blow-up profile for small radial type II blow up solutions of the energy critical wave equation}, J.\ Eur.\ Math.\ Soc.~13 (2011),  533--599.

\bibitem{DKM}  T.\ Duyckaerts, C.\ Kenig, F.\ Merle, \emph{Classification of radial solutions of the focusing, energy-critical wave equation}, Cambridge Journal of Mathematics
Volume 1, Number 1 (2013), 75--144. 

\bibitem{DKMnonradial} T.\ Duyckaerts, C.\ Kenig, F.\ Merle, \emph{Universality of the blow-up profile for small type II blow-up solutions of energy-critical wave equation: the non-radial case}, J.\ Eur.\ Math.\ Soc. 14 (2012), no.~5, 1389--1454.

\bibitem{DKM3}  T.\ Duyckaerts, C.\ Kenig, F.\ Merle, \emph{Solutions of the focusing nonradial critical wave equation with the compactness property}, to appear in Ann.\ Sc.\ Norm.\ Super.\ Pisa Cl.\ Sci., see also arXiv:1402.0365 

\bibitem{GV} J.\ Ginibre, G.\ Velo, \emph{Generalized Strichartz inequalities for the wave equation}, 
J.\ Funct.\ Anal.\ 1995, Vol.\ 133, no.~1, 50--68.


\bibitem{Gri1} M.\ Grillakis, \emph{Global regularity and asymptotic behavior of the wave equation with a critical nonlinearity}, Ann.\ of.\ Math.~132 (1990), 485--509.

\bibitem{Gri2} M.\ Grillakis, \emph{Regularity of the wave equation with a critical nonlinearity}, Comm.\ Pure.\ Appl.\ Math, 45 (1992), 749--774.

\bibitem{JiaLiuXu}  H.\ Jia, B.\ P.\ Liu, G.\ X.\ Xu, \emph{Long time dynamics of defocusing energy critical $3 + 1$ dimensional wave equation with potential in the radial case}, To Appear in Comm.\ Math.\ Phys., arXiv:1403.5696 


\bibitem{KT}  M.\ Keel, T.\ Tao, \emph{Endpoint Strichartz estimates},
Amer.\ J.\ Math., 1998,  Vol. 120, no.~5, 955--980. 


\bibitem{KLS} C.\ Kenig, A.\  Lawrie, W.\  Schlag, \emph{Relaxation of wave maps exterior to a ball to harmonic maps for all data},  Geometric and Functional Analysis (GAFA). 24 (2014), no.~2, 610--647.  

\bibitem{KLIS2} C.\ Kenig, A.\ Lawrie, B.\ P.\ Liu, W.\ Schlag, \emph{Stable soliton resolution for exterior wave maps in all equivariance classes},  arXiv:1409.3644. 

\bibitem{KM} S.\ Klainerman, M.\  Machedon, \emph{Space-time estimates for null forms and the local existence theorem}, 
Comm.\ Pure Appl.\ Math.\ 1993, Vol.\ 46,  no.~9,1221--1268.




\bibitem{SchNak} K.\ Nakanishi, W.\ Schlag, \emph{Invariant manifolds and dispersive hamiltonian evolution equations}, Z\"urich Lectures in Advanced Mathematics, EMS, 2011. 

\bibitem{NS}  K.\ Nakanishi, W.\ Schlag, \emph{Global dynamics above the ground state energy for the 
focusing nonlinear Klein-Gordon equation}, J.\ Differential Equations 2011, Vol.~250, no.~5, 
2299--2333. 







\bibitem{Stein} E.\ M.\  Stein, Harmonic Analysis, Princeton University Press, 1993. 

\bibitem{SoffWein} A.\ Soffer, M.\ I.\  Weistein, \emph{Resonances, radiation damping and instability in Hamiltonian nonlinear wave equations}, Invent.\ Math.\ 136 (1999), 9--74.

\bibitem{Struwe} M.\ Struwe, \emph{Globally regular solutions to the $u^5$ Klein-Gordon equation}, Ann.\ Scuola Norm.\ Sup.\ Pisa Cl.\ Sci.\ (4) 15 (1988), no.~3, 495--513. 


\bibitem{Tao} T.\ Tao, \emph{Spacetime bounds for the energy-critical nonlinear wave equation in three spatial dimensions}, Dyn.\ Partial Differ.\ Equ. 3 (2006), no.~2, 93--110. 


\bibitem{TaoCompact} T.\ Tao, \emph{A global compact attractor for high-dimensional defocusing non-linear Schr\"odinger equations with potential}, Dyn.\ Partial Differ.\ Equ.\ 5 (2008), no.~2, 101--116.

\bibitem{Taostri} T.\ Tao, \emph{Spherically averaged endpoint Strichartz estimates for the two-dimensional Schr\"odinger equation}, Comm.\  Partial Differential Equations 2000, Vol.\ 25, no. 7--8, 
1471--1485. 


\bibitem{TsaiYau} T.\ P.\  Tsai, H.\ T.\  Yau, \emph{Asymptotic dynamics of nonlinear Schr\"odinger equations: resonance-dominated and dispersion-dominated solutions}. 
Comm.\ Pure Appl.\ Math.\ 55 (2002), no.~2, 153--216. 

\bibitem{Yajima} K.\ Yajima, \emph{The $W^{k,p}$-continuity of wave operators for Schr\"odinger operators}, J.\ Math.\ Soc.\ Japan 47 (1995), no.~3, 551--581. 











\end{thebibliography}
\end{document}